\documentclass[11pt]{amsart}
\usepackage{amsmath,amssymb,amsthm,enumerate,xypic}
\usepackage[hmargin=25mm, vmargin=20mm]{geometry}
\usepackage{xy} 
\usepackage{mathrsfs} 
\usepackage{xspace}
\usepackage[colorlinks]{hyperref}

\title[Orientation of piecewise powers of a minimal homeomorphism]{Orientation of piecewise powers of a minimal homeomorphism}
\author{Colin D. Reid}
\address{University of Newcastle\\School of Information and Physical Sciences\\University Drive\\Callaghan NSW 2308 Australia}

\newtheorem{thm}{Theorem}[section]

\newtheorem{prop}[thm]{Proposition}
\newtheorem{lem}[thm]{Lemma}
\newtheorem{cor}[thm]{Corollary}

\theoremstyle{definition}
\newtheorem{defn}[thm]{Definition}
\newtheorem{qu}[thm]{Question}
\newtheorem{rmk}[thm]{Remark}

\newcommand{\bZ}{\mathbb{Z}}
\newcommand{\bN}{\mathbb{N}}
\newcommand{\bR}{\mathbb{R}}

\newcommand{\mc}[1]{\mathcal{#1}}

\newcommand{\pap}{p.a.p.\@\xspace}
\newcommand{\ppm}{p.p.m.\@\xspace}

\newcommand{\Homeo}{\mathrm{Homeo}}

\newcommand{\N}{\mathrm{N}}
\newcommand{\CC}{\mathrm{C}}

\newcommand{\inv}{^{-1}}
\newcommand{\supp}{\mathrm{supp}}

\newcommand{\id}{\mathrm{id}}

\newcommand{\per}{\mathrm{per}}
\newcommand{\aper}{\mathrm{aper}}

\newcommand{\defbold}{\textbf}

\newcommand{\triv}{\{1\}}

\newcommand{\grp}[1]{\langle #1 \rangle}

\newcommand{\ol}[1]{\overline{#1}}

\newcommand{\Fgp}[1]{\tau[#1]}
\newcommand{\Fgpp}[1]{\tau_{\mathbb{N}}[#1]}
\newcommand{\pipp}{\pi_{\mathbb{N}}}
\newcommand{\Fgpnn}[1]{\tau_{+}[#1]}

\begin{document}

\begin{abstract}
We show that given a compact minimal system $(X,g)$ and an element $h$ of the topological full group $\Fgp{g}$ of $g$, then the infinite orbits of $h$ admit a locally constant orientation with respect to the orbits of $g$.  We use this to obtain a clopen partition of $(X,G)$ into minimal and periodic parts, where $G$ is any virtually polycyclic subgroup of $\Fgp{g}$.  We also use the orientation of orbits to give a refinement of the index map and to describe the role in $\Fgp{g}$ of the submonoid generated by the induced transformations of $g$.  Finally, we consider the problem, given a homeomorphism $h$ of the Cantor space $X$, of determining whether or not there exists a minimal homeomorphism $g$ of $X$ such that $h \in \Fgp{g}$.
\end{abstract}

\maketitle

\section{Introduction}

Let $X$ be a Hausdorff topological space and let $\Homeo(X)$ be the group of homeomorphisms from $X$ to itself; write $\grp{g} := \{g^n \mid n \in \bZ\}$ for the group generated by $g \in \Homeo(X)$.  We say $g$ is \defbold{minimal} if $X$ is nonempty, and whenever $K$ is a proper closed subspace of $X$ such that $gK \subseteq K$, then $K = \emptyset$.  Given a subgroup $G$ of $\Homeo(X)$, the \defbold{topological full group} $\Fgp{G}$ of $G$ consists of all homeomorphisms $h$ of $X$ such that there is a locally constant function $c: X \rightarrow G$, called the \defbold{cocycle} of $h$ with respect to $g$, such that $hx = c(h)x$ for all $x \in X$.  In other words, $\Fgp{G}$ consists of those homeomorphisms that are `piecewise' in $G$ with clopen pieces.  More specifically, given $g \in \Homeo(X)$, then we write $\Fgp{g} := \Fgp{\grp{g}}$.  Note that if $g$ is aperiodic (that is, $g$ has no finite orbits), then the cocycle of $h$ with respect to $\grp{g}$ is unique: we write it as a locally constant function $c_{g,h}: X \rightarrow \bZ$ so that $hx = g^{c_{g,h}}x$ for all $x \in X$.

The topological full group was introduced several times independently in the 1980s, first by W. Krieger \cite{Krieger} in 1980, then by P. \v{S}t\v{e}p\'{a}nek and M. Rubin \cite{SR} in 1989, and also with a focus on analysing a Cantor minimal system (that is, the Cantor space equipped with a single homeomorphism acting minimally) by I. Putnam \cite{Putnam} in 1989.  The study of Cantor minimal systems in particular took off in the 1990s with further work of T. Giordano, Putnam and C. Skau \cite{GPS1}, \cite{GPS2}.  (See also \cite{Tom}.)  Topological full groups also turn out to have remarkable properties from a group-theoretic perspective, providing for instance the first known examples of infinite finitely generated simple amenable groups (see \cite{JM}).  Their theory has been developed and generalized by many different authors, for example to the setting of \'{e}tale groupoids; see \cite{Cor} and \cite{Mat} for surveys of some recent developments.

The definition of the topological full group makes sense for quite general classes of action or partial action by homeomorphisms.  However, for this article we focus on a structure that is particularly natural to define with respect to a single aperiodic homeomorphism $g$, namely the partial order generated by setting $x \le gx$ for all $x \in X$.  In this context we identify two types of `positive' element of $\Fgp{g}$ and derive consequences for the structure of general elements of $\Fgp{g}$.

\begin{defn}
Given a compact Hausdorff space $X$ and an aperiodic homeomorphism $g$ of $X$, define a partial order $\le_g$ on $X$ by setting $x \le_g y$ if $y = g^tx$ for some $t \ge 0$.  All orbits of homeomorphisms are two-sided unless otherwise specified.  Given $h \in \Fgp{g}$, we say an $h$-orbit $\Omega$ is \defbold{positive} (with respect to $g$) if for all $y,z \in \Omega$, there is $n \in \bN$ such that $h^{n'}y >_g z >_g h^{-n'}y$ for all $n' \ge n$; \defbold{strongly positive} (with respect to $g$) if for all $y \in \Omega$ we have $hy >_g y$, or in other words, $c_{g,h}(y) > 0$; and \defbold{trivial} if $|\Omega|=1$.  The orbit is \defbold{(strongly) negative} with respect to $g$ if it is (strongly) positive with respect to $g\inv$.  Say that $h$ is \defbold{(strongly) positive} (with respect to $g$) if its nontrivial orbits are all (strongly) positive with respect to $g$; write $\Fgpnn{g}$ for the set of positive elements of $\Fgp{g}$ with respect to $g$ and $\Fgpp{g}$ for the set of strongly positive elements with respect to $g$.

A \defbold{compact minimal system} $(X,g)$ consists of an infinite compact Hausdorff space $X$ equipped with a minimal (hence also aperiodic) homeomorphism $g$.
\end{defn}

The approach of looking at orientations of orbits was suggested to me by F. Le Ma\^{i}tre, who was in turn inspired by work of R. Belinskaya \cite{Bel}, both working in the measure-theoretic setting.  Several of the results in this article will mimic those of \cite[\S4.3]{LM} in particular.  The novelty in the topological dynamics setting is to establish that various sets of interest are actually clopen.

Our first main result is that given a compact minimal system $(X,g)$, every element of the topological full group can be naturally partitioned into a clopen positive, negative and periodic part.

\begin{thm}[See \S\ref{sec:cyclic-sign}]\label{thm:sign}
Let $(X,g)$ be a compact minimal system and let $h \in \Fgp{g}$.  Then $X$ admits a partition into $\grp{h}$-invariant clopen subspaces
\[
X = X_+ \sqcup X_- \sqcup X_\per,
\]
where $X_+$ is the union of positive orbits and $X_-$ is the union of negative orbits of $h$, and where $h$ has finite order on $X_\per$.
\end{thm}

We refer to the partition of the action of $h \in \Fgp{g}$ given by Theorem~\ref{thm:sign} as the \defbold{sign partition} of $h$ (with respect to $g$).  Closely related to the existence of the sign partition of $h$ is another kind of partition that represents an intrinsic property of the action of $\grp{h}$ on $X$ (that is, without any direct reference to $g$), and is analogous to a phenomenon observed by M. Keane (\cite{Keane}) in the context of interval exchange transformations.\footnote{Keane's results are not stated in exactly this form, however if one looks at the results in \cite[\S3]{Keane}, one deduces that the finite orbits of an arbitrary interval exchange transformation $T$ of the unit interval have bounded size, and that for $T$ aperiodic but not minimal, then $[0,1]$ can be decomposed into $T$-invariant parts that are each finite unions of intervals, such that after reordering, $T$ has fewer breakpoints on each part than on $[0,1]$ as a whole.  A minimal-periodic partition of $[0,1]$ then follows by induction, where the parts are finite unions of intervals.}In fact, the next theorem applies not just to cyclic subgroups, but all virtually polycyclic subgroups of $\Fgp{g}$.  (A group $G$ is \defbold{polycyclic} if $G$ admits a subnormal series with finitely many factors, each of which is cyclic; $G$ is \defbold{virtually polycyclic} if it has a polycyclic subgroup of finite index.)

\begin{defn}
Let $G$ be a group of homeomorphisms of a compact Hausdorff space $X$; write $X_{\per}$ for the union of the finite orbits of $G$ and $X_{\aper}$ for the union of the infinite orbits of $G$.  We say $G$ is \defbold{strongly periodic} if there are finitely many distinct subgroups $K_1,\dots,K_n$ of $G$ of finite index, such that for each $1 \le j \le n$ the set of points $X_\per(j)$ with stabilizer conjugate to $K_j$ is clopen and $X = \bigsqcup^n_{j=1}X_\per(j)$.  Equivalently, $G$ is strongly periodic if and only if it acts with kernel of finite index and for every $h \in G$ the set of fixed points of $h$ on $X$ is clopen.

More generally, we say $G$ admits a \defbold{minimal-periodic partition} if there is a partition of $X$ into clopen $G$-invariant subspaces
\[
X = \bigsqcup^m_{i=1}X_\aper(i) \sqcup X_\per,
\]
where $G$ acts minimally on each of the infinite subspaces $X_\aper(1), \dots, X_\aper(m)$ and has strongly periodic action on $X_\per$.  Note that if a minimal-periodic partition exists, it is unique up to reordering of the subspaces $X_\aper(i)$ and $X_\per(j)$.  We make the same definitions for an individual homeomorphism $h \in \Homeo(X)$ using $G = \grp{h}$.
\end{defn}

\begin{thm}[See \S\ref{sec:minimal-periodic}]\label{thm:minimal_periodic}
Let $X$ be an infinite compact Hausdorff space, let $g \in \Homeo(X)$ and let $G$ be a virtually polycyclic subgroup of $\Fgp{g}$.  If $g$ admits a minimal-periodic partition, then so does $G$.

Moreover, if $g$ admits a minimal-periodic partition, then for each closed $G$-minimal subspace $Y$ of $X$ there is a normal subgroup $N$ of $G$ and a partition of $Y$ into finitely many clopen $N$-minimal spaces $Y_1,\dots,Y_n$, such that, writing $N_i$ for the kernel of the action of $N$ on $Y_i$, then $N/N_i \cong \bZ$.
\end{thm}

In particular, any compact minimal system $(X,g)$ clearly admits a minimal-periodic partition, so Theorem~\ref{thm:minimal_periodic} says in this case that every virtually polycyclic $G \le \Fgp{g}$ admits a minimal-periodic partition.  In the case that $G = \grp{h}$, the minimal-periodic partition of $h$ refines the sign partition of $h$: each of the subspaces $X_\per(j)$ is contained in the periodic part $X_\per$, and each of the aperiodic minimal parts $X_\aper(i)$ of $h$ is contained in either $X_+$ or $X_-$.  There is also the following application to pointwise almost periodic (\pap) groups of homeomorphisms, or equivalently, groups $G$ of homeomorphisms such that each orbit closure is a compact minimal $G$-set (see Definition~\ref{def:pap} and Lemma~\ref{lem:pap_char}).

\begin{cor}\label{cor:pap}
Let $X$ be a locally compact Hausdorff space and let $f \in \Homeo(X)$.  If $\grp{f}$ is \pap on $X$, then so is every virtually polycyclic subgroup $G$ of $\Fgp{f}$.
\end{cor}

As the name suggests, every strongly positive element is positive, but not conversely.  The relationship between $\Fgpnn{g}$ and $\Fgpp{g}$ can be summarized as follows.

\begin{thm}[See \S\ref{sec:positive}]\label{thm:positive_conjugation}
Let $(X,g)$ be a compact minimal system.
\begin{enumerate}[(i)]
\item $\Fgpnn{g}$ is closed under conjugation in $\Fgp{g}$; moreover, given $h \in \N_{\Homeo(X)}(\Fgp{g})$, then $h\Fgpnn{g}h\inv \in \{\Fgpnn{g},(\Fgpnn{g})\inv\}$.
\item Given $h \in \Fgpnn{g}$, there is a unique $\Fgp{h}$-conjugate $h'$ of $h$ such that $h' \in \Fgpp{g}$.
\item We have $\bigcap_{h \in \Fgpp{g}}h\Fgpp{g}h\inv = \triv$.
\end{enumerate}
\end{thm}

Given a compact minimal system $(X,g)$, it is well-known that there is a unique group homomorphism $I: \Fgp{g} \rightarrow \bZ$ such that $I(g)=1$, called the index map.  This map was implicitly used by \v{S}t\v{e}p\'{a}nek and Rubin (see the proof of Lemma~5.17 in \cite{SR}); an equivalent definition is also given more explicitly in \cite[\S3.3]{Cor}.  The index map was rediscovered in \cite{GPS2} using an integral formula.

Continuing the theme of orientation of orbits, given $h \in \Fgp{g}$ we can split $I(h)$ into a positive part and a negative part, which count respectively the number of positive and negative $h$-orbits inside a $g$-orbit.

\begin{thm}[See Propositions~\ref{prop:orbit_number} and~\ref{prop:index_sum}]\label{thm:index}
Let $(X,g)$ be a compact minimal system and let $h \in \Fgp{g}$.  Then there are nonnegative integers $o^+(h)$ and $o^-(h)$, such that every $g$-orbit contains exactly $o^+(h)$ positive $h$-orbits and $o^-(h)$ negative $h$-orbits.  Moreover, the index map is given by
\[
I(h) = o^+(h) - o^-(h).
\]
\end{thm}

Write $\mc{CO}^*(X)$ for the set of nonempty compact open subsets of $X$.  Given $A \in \mc{CO}^*(X)$, the \defbold{induced transformation} $g_A \in \Homeo(A)$ is defined by setting, for each $x \in A$, the image $g_Ax$ to be $g^tx$, where $t$ is the least positive integer such that $g^tx \in A$.  We also regard $g_A$ as a homeomorphism of $X$ by setting $g_Ax = x$ for all $x \in X \smallsetminus A$.   Notice that $I(g_A) = 1$ while $I(h) > 0$ for all $h \in \Fgpp{g} \smallsetminus \{1\}$, so $g_A$ cannot be decomposed as a product of $\ge 2$ nontrivial elements of $\Fgpp{g}$.  On the other hand the induced transformations form a generating set for both $\Fgpp{g}$ (as a monoid) and $\Fgp{g}$ (as a group), with respect to which we can put elements into a normal form.

\begin{thm}[See \S\ref{sec:normal_form:induced}]\label{thm:normal_form}
Let $(X,g)$ be a compact minimal system.  Let $h \in \Fgp{g}$.  Then $h$ can be expressed as follows in exactly one way:
\[
h =  g_{A_n} \dots g_{A_2} g_{A_1} g^r \text{ such that } r \in \bZ, \; A_i \in \mc{CO}^*(X) \; \text{ and } A_n \subseteq \dots \subseteq A_2 \subseteq A_1 \subsetneq X.
\]
Moreover, $h \in \Fgpp{g}$ if and only if $r \ge 0$.
\end{thm}

Under the hypotheses of Theorem~\ref{thm:normal_form}, if $X$ is zero-dimensional, it can be deduced from Rubin's spatial reconstruction theorem \cite[Theorem~2.14]{Rubin} that the group $\Fgp{g}$ is a complete invariant for the flip conjugacy class of $(X,g)$; this was done more explicitly for Cantor minimal systems in \cite{GPS2}.  Analogously, the monoid $\Fgpp{g}$ is a complete invariant for the conjugacy class of $(X,g)$, and it is relatively easy to see how the space $X$ manifests in the monoid structure (see Proposition~\ref{prop:positive:reconstruction}).  Theorem~\ref{thm:normal_form} also yields a presentation of $\Fgpp{g}$, which is equivalent to specifying a certain binary operation on the set of nonempty compact open subsets of $X$ (see Proposition~\ref{prop:monoid_presentation}).

\

Say that a homeomorphism $h$ of a compact Hausdorff space $X$ is \defbold{piecewise a power of a minimal homeomorphism} (\ppm) if $X = \emptyset$ or there exists $g \in \Homeo(X)$ such that $g$ is minimal on $X$ and $h \in \Fgp{g}$.  As we see from Theorem~\ref{thm:minimal_periodic}, \ppm homeomorphisms have a special form, namely they admit a minimal-periodic partition.  Now suppose that we are given some homeomorphism $h$ of a compact Hausdorff space that admits a minimal-periodic partition.    Then it is natural to ask, what conditions on $h$ are necessary or sufficient to construct $g \in \Homeo(X)$ such that $h \in \Fgp{g}$?  We obtain a partial answer to this question.

Observe that if $h$ is \ppm, then so are all its induced transformations on nonempty clopen subspaces: specifically, if $h \in \Fgp{g}$ where $g$ is minimal, and $A$ is a nonempty clopen subspace, then $h_A \in \Fgp{g_A}$.  Conversely, on the Cantor space $X$ (or more generally if $X$ is a minimal h-homogeneous space, see Section~\ref{sec:ppm}), then $(X,h)$ is \ppm if and only if $h$ admits a minimal-periodic partition and either $X_\aper$ is empty, or $(X_\aper,h)$ is \ppm (see {Proposition~\ref{prop:remove_periodic}).  Thus for the Cantor space, the problem of characterizing when the homeomorphism $h$ is \ppm reduces to the case that $h$ is aperiodic.

Given a \ppm homeomorphism $h \in \Homeo(X)$, and given a minimal homeomorphism $g$ such that $h \in \Fgp{g}$, we can take the orbit number $o_g(h) := o^+_g(h) + o^-_g(h)$ as a measure of the `efficiency' with which $g$ witnesses that $h$ is p.p.m.  Write $o_{\min}(h)$ for the least value of $o_g(h)$, as $g$ ranges over all minimal homeomorphisms of $X$ such that $h \in \Fgp{g}$.  It is easily seen that $o_{\min}(h)$ is at least the number $m(h)$ of infinite minimal orbit closures of $h$; say that $h$ is \defbold{strongly \ppm} if $o_{\min}(h) = m(h)$.  We can characterize the aperiodic strongly \ppm homeomorphisms using a notion of equivalence between the infinite minimal orbit closures.  We say two compact minimal systems $(X_1,h_1)$ and $(X_2,h_2)$ are \defbold{(flip) Kakutani equivalent} if there are nonempty clopen subspaces $Y_i \subseteq X_i$ such that the induced systems $(Y_1,(h_1)_{Y_1})$ and $(Y_2,(h_2)_{Y_2})$ are (flip) conjugate.  The concept of Kakutani equivalence was introduced in the ergodic setting by S. Kakutani, then translated to topological dynamics by later authors (see \cite{Pet}).

\begin{thm}[{See Proposition~\ref{prop:Kakutani_weld:bis}}]
Let $X$ be a compact Hausdorff space and let $h \in \Homeo(X)$ be aperiodic.  Then the following are equivalent:
\begin{enumerate}[(i)]
\item $h$ is strongly \ppm;
\item There is a partition of $X$ into clopen spaces $X_1,\dots,X_m$ such that $(X_1,h),\dots,(X_m,h)$ are compact minimal systems that lie in a single flip Kakutani equivalence class.
\end{enumerate}
\end{thm}

We can thus restate the \ppm property for the Cantor space as follows.  Say that the tuple $(X_i,h_i)_{1 \le i \le m}$ of compact minimal systems is \defbold{Kakutani compatible} if there exists a family of compact minimal systems $(X_1,g_1),\dots, (X_m,g_m)$, all lying in a single Kakutani equivalence class, such that $h_i \in \Fgp{g_i}$.

\begin{cor}
Let $X$ be the Cantor space and let $h \in \Homeo(X)$.  Then the following are equivalent:
\begin{enumerate}[(i)]
\item $h$ is \ppm;
\item There is a partition of $X$ into clopen spaces $X_0,X_1,\dots,X_m$ such that $h$ is strongly periodic on $X_0$ and $(X_1,h),\dots,(X_m,h)$ are compact minimal systems such that $(X_i,h_i)_{1 \le i \le m}$ is a Kakutani compatible tuple.
\end{enumerate}
\end{cor}

\subsection*{Open questions}

Theorem~\ref{thm:sign} ensures in particular that for $X$ a compact Hausdorff space, if $g \in \Homeo(X)$ is minimal, then for all $h \in \Fgp{g}$, the set $X^{\grp{h}}_{\aper}$ is clopen and in the action of $h$ on $X^{\grp{h}}_{\aper}$, the orientation of $h$-orbits relative to $g$ is locally constant.  Some of this structure is present for $h \in \Fgp{g}$ if we only assume $g \in \Homeo(X)$ aperiodic: see Lemma~\ref{lem:aperiodic_drift} below.  In this context we cannot expect a minimal-periodic partition for $h$, but it is natural to ask under what conditions $h$ admits (a weak form of) the sign partition with respect to $g$.

\begin{qu}\label{que:aperiodic}
Let $X$ be a compact Hausdorff space.  For which aperiodic $g \in \Homeo(X)$ do we have either or both of the following?
\begin{enumerate}[(a)]
\item For all $h \in \Fgp{g}$, $X^{\grp{h}}_{\aper}$ is open.  (Equivalently, the finite $h$-orbits have bounded length.)
\item For all $h \in \Fgp{g}$, then $X^{\grp{h}}_{\aper}$ can be partitioned into compact sets $X_+$ and $X_-$, such that $h$ is positive on $X_+$ and negative on $X_-$.  (Equivalently: for all $h \in \Fgp{g}$, the set of points $x$ such that $c_{g,h^n}(x)$ tends to $+\infty$ as $n$ tends to $+\infty$ is clopen in $X^{\grp{h}}_{\aper}$.)
\end{enumerate}
\end{qu}

In Theorem~\ref{thm:minimal_periodic}, the assumption that $G$ is virtually polycyclic was needed for the specific approach of the proof, but it is interesting to ask whether such minimal-periodic partitions exist more generally for finitely generated subgroups of $\Fgp{g}$, in the case that $g$ admits a minimal-periodic partition.  (Finite generation at least ensures that the set of fixed points of $G$ is clopen, see Lemma~\ref{lem:periodic}; without this condition, there are certainly subgroups without a minimal-periodic partition, such as a point stabilizer in $\Fgp{g}$).  Certainly, in this situation $G$ has a bounded number of infinite orbits on each $g$-orbit (Lemma~\ref{lem:infinite_orbit_bound}).  Note that not all finitely generated groups can occur in this context: by \cite[Theorem~A]{JM}, $G \le \Fgp{g}$ must be amenable.

\begin{qu}
Which finitely generated groups $G$ have the following property?

Let $(X,g)$ be a compact minimal system and let $\pi: G \rightarrow \Fgp{g}$ be a homomorphism.  Then $(X,\pi(G))$ admits a minimal-periodic partition.
\end{qu}

One can also ask for a version of a minimal-periodic partition in the context of minimal one-parameter flows, that is, we have a compact Hausdorff space $X$ and a continuous homomorphism $g: \bR \rightarrow \Homeo(X)$, such that $\{g(r).x \mid r \in \bR\}$ is dense for every $x \in X$.  Note that in this case $X$ is connected, so we have to replace clopen partitions with some generalization.  Suppose $(X,g)$ is a minimal one-parameter flow.  One way to define a `piecewise group' $\Fgp{g}$ of $(X,g)$ in this context would be the group of bijections $h$ of $X$, such that there is a partition $P = \{X_1,\dots,X_n\}$ of $X$ into finitely many pieces and $r_1,\dots,r_n \in \bR$, such that $hx = g(r_i)x$ for all $x \in X_i$, and such that the union of the interiors of the sets $X_i$ is dense in $X$.  (The group of interval exchange transformations is a special case of this, where one takes $X$ to be the circle group and $g(r) = (e^{2\pi i t} \mapsto e^{2\pi i(t+r)})$.)

\begin{qu}
Let $(X,g)$ be a minimal one-parameter flow and let $h \in \Fgp{g}$.  Are there disjoint $\grp{h}$-invariant open subsets $X_0,X_1,\dots,X_m$ of $X$, with dense union, such that $h$ has finite order on $X_0$ and acts minimally on $X_i$ for $1 \le i \le m$?
\end{qu}

The Kakutani equivalence classes of minimal homeomorphisms of the Cantor space were classified by \cite[Theorem~4.7]{HPS} and \cite[Theorem~3.8]{GPS1}: they are given by commensurability classes of simple ordered Bratteli diagrams, where two ordered Bratteli diagrams are commensurable if one can be obtained from the other by telescoping and changing finitely many edges.  What is not clear is how to describe, given a Cantor minimal system $(X,h)$, the ordered Bratteli diagrams of those systems $(X,g)$ for which $h \in \Fgp{g}$ in terms of that of $(X,h)$.  In particular, it is not clear if Kakutani compatibility reduces to an equivalence relation defined on pairs of spaces.

\begin{qu}
Let $(X_1,h_1),\dots,(X_m,h_m)$ be Cantor minimal systems such that for all ${1 \le i \le m-1}$ the pair $((X_i,h_i),(X_{i+1},h_{i+1}))$ is Kakutani compatible.  Is $(X_i,h_i)_{1 \le i \le m}$ Kakutani compatible?
\end{qu}

\subsection*{Structure of the article} After a preliminary section (\S\ref{sec:prelim}), we establish the existence of the sign partition and the minimal-periodic partition (\S\ref{sec:sign}).  We then establish the key properties that relate the set of positive elements to the set of strongly positive elements (\S\ref{sec:positive}).  In \S\ref{sec:index} we introduce orbit numbers and relate them to the index map.  In the next section (\S\ref{sec:normal_form:induced}) we prove the normal form for elements of $\Fgp{g}$ in terms of induced transformations, leading to a presentation for $\Fgpp{g}$; for the latter we also obtain an easy reconstruction of the space from the monoid.  The final section (\S\ref{sec:ppm}) gives partial results on the problem of determining whether a given homeomorphism is piecewise a power of a minimal homeomorphism.

\subsection*{Acknowledgements}  Research funded by ARC project FL170100032.  I thank Yves Cornulier and Fran\c{c}ois Le Ma\^{i}tre for very helpful comments and suggestions.  I also thank the anonymous referee, who has suggested a number of improvements to the article.

\section{Preliminaries}\label{sec:prelim}

\subsection{Notation}

In this section we set some notation and recall some standard concepts that will be used throughout.

For this article, $0 \in \bN$.

Given a function $\alpha: X \rightarrow Y$, and $x \in X$, we will simply write $\alpha x$ to mean $\alpha(x)$, where there is no danger of confusion.  Since composition of functions is associative, we can similarly write $\alpha_n \alpha_{n-1} \dots \alpha_1 x$ to mean that the sequence $\alpha_1,\dots,\alpha_n$ of functions is applied successively to $x$.  Given a subset $K$ of $X$, we define $\alpha K := \{\alpha x \mid x \in K\}$, and given a set of $S$ of functions defined at a point $x$, we define $Sx := \{sx \mid s \in S\}$.

Given a topological space $X$, we write $\mc{CO}(X)$ for the set of compact open subsets of $X$ and $\mc{CO}^*(X) = \mc{CO}(X) \smallsetminus \{\emptyset\}$.  Note that if $X$ is compact Hausdorff, then $\mc{CO}(X)$ is just the set of clopen subsets of $X$.  We say that a locally compact Hausdorff space $X$ is \defbold{zero-dimensional} if $\mc{CO}(X)$ is a base for the topology of $X$.

By a \defbold{homeomorphism of $X$} we mean a homeomorphism from $X$ to itself.  We say a homeomorphism or group of homeomorphisms is \defbold{aperiodic} if it has no finite orbits.  Given a group of homeomorphisms $G$ of the compact Hausdorff space $X$, write $X^G_{\per}$ for the union of all finite $G$-orbits on $X$ and $X^G_{\aper}$ for the union of all infinite $G$-orbits on $X$.  We will drop the superscript when the acting group is clear from context.

\subsection{The cocycle}\label{sec:cocycle}

Let $X$ be an infinite compact Hausdorff space, let $g$ be an aperiodic homeomorphism of $X$ and let $h \in \Fgp{g}$.  Then for each $x \in X$ there is exactly one $t \in \bZ$ such that $hx = g^tx$.  This defines a continuous map, the \defbold{cocycle of $h$ with respect to $g$}, which is described as follows:
\[
c_{g,h}: X \rightarrow \bZ; \quad \forall x \in X: hx = g^{c_{g,h}(x)}x.
\]

Since the cocycle is a continuous map from a compact space to a discrete space, it takes only finitely many values.  We define $|h|_g := \max \{ |c_{g,h}(x)| \mid x \in X\}$.  We also define a partial order $\le_g$ on $X$: say $x \le_g y$ if $y = g^tx$ for some $t \ge 0$.

Given a finitely generated $G \le \Fgp{g}$, if the stabilizer of $x \in X$ has finite index, then it depends continuously on $x \in X$.

\begin{lem}\label{lem:periodic}
Let $X$ be a compact Hausdorff space and let $g$ be an aperiodic homeomorphism of $X$.  Let $G$ be a finitely generated subgroup of $\Fgp{g}$ and let $H$ be a finite index subgroup of $G$.  Then the set $Y := \{x \in X \mid G_x = H\}$ is clopen in $X$.  In particular, $X^G_{\aper}$ is closed, and if $X^G_\per$ is closed then $(X^G_\per,G)$ is strongly periodic.
\end{lem}

\begin{proof}
Let $S$ be a finite generating set for $H$ (which exists since $H$ has finite index in the finitely generated group $G$), and let $F$ be a set of representatives for the left cosets of $H$ in $G$.

We can obtain $Y$ as
\[
Y = \{x \in X \mid \forall h \in S: c_{g,h}(x) = 0; \forall h \in F \smallsetminus H: c_{g,h}(x) \neq 0\}.
\]
Thus the condition $x \in Y$ is defined by constraints on $c_{g,h}(x)$ for finitely many $h \in G$; it follows that $Y$ is clopen.

Finally, we see that $X^G_{\aper}$ is closed, because we can write it as the following intersection of closed sets:
\[
X^G_{\aper} := \bigcap_{H \le G, |G:H| < \infty} \left( X \smallsetminus \{x \in X \mid G_x = H\} \right). \qedhere
\]
\end{proof}

\subsection{Minimality and pointwise almost periodicity}

Given a compact Hausdorff space $X$, a subgroup $G$ of $\Homeo(X)$ and a subset $Y \subseteq X$, we say $Y$ is \defbold{(strongly) $G$-invariant} if $gY = Y$ for all $g \in G$, and $Y$ is \defbold{$G$-minimal} (or $G$ is minimal on $Y$) if $Y$ is nonempty, $G$-invariant, and the only proper closed $G$-invariant subset of $Y$ is the empty set.

The following is immediate from compactness and Zorn's lemma.

\begin{lem}\label{lem:minimal_exists}
Let $X$ be a compact Hausdorff space, let $G$ be a subgroup of $\Homeo(X)$ and let $K$ be a nonempty closed $G$-invariant subspace of $X$.  Then $K$ contains a $G$-minimal subspace.
\end{lem}

On a compact Hausdorff space, if a cyclic subgroup $\grp{g} \le \Homeo(X)$ is minimal, then in fact all forward orbits of $g$ are dense, in other words there is no proper nonempty closed subset $K$ such that $gK \subseteq K$.

\begin{lem}\label{lem:minimal_forwards}
Let $X$ be a nonempty compact Hausdorff space and let $g \in \Homeo(X)$.  Then the following are equivalent:
\begin{enumerate}[(i)]
\item whenever $K$ is a proper closed subspace of $X$ such that $gK \subseteq K$, then $K = \emptyset$;
\item whenever $K$ is a proper closed subspace of $X$ such that $gK = K$, then $K = \emptyset$.
\end{enumerate}
\end{lem}

\begin{proof}
Clearly (i) implies (ii).  Conversely, suppose that (ii) holds, and let $K$ be a proper closed subspace of $X$ such that $gK \subseteq K$.  We see that in fact $g^mK \subseteq g^nK$ whenever $m \ge n \ge 0$.  Set $L = \bigcap_{n \ge 0}g^nK$.  Then $L$ is compact and $gL = \bigcap_{n \ge 1}g^nK = L$ (here we use the fact that $g$ is injective); since $L \subseteq K \neq X$, it follows that $L = \emptyset$.  By compactness we must have $g^nK = \emptyset$ for some $n$, so in fact $K = \emptyset$, proving (i).
\end{proof}

Locally minimal actions can be defined in a few equivalent ways, as in the next proposition.

\begin{prop}\label{prop:locally_minimal}
Let $X$ be a nonempty compact Hausdorff space, let $G$ be a subgroup of $\Homeo(X)$ and let $Y \subseteq X$ be closed and $G$-invariant.  Then the following are equivalent:
\begin{enumerate}[(i)]
\item every closed $G$-minimal subset of $Y$ has nonempty interior in $X$;
\item every closed $G$-invariant subset of $Y$ is open in $X$;
\item $Y$ is the disjoint union of finitely many $G$-minimal clopen subspaces of $X$.
\end{enumerate}
\end{prop}

\begin{proof}
The implications (iii) $\Rightarrow$ (ii) $\Rightarrow$ (i) are all clear.  Suppose now that (i) holds, and let $U$ be the union of the $X$-interiors of all closed $G$-minimal subspaces of $Y$.  Then by applying Lemma~\ref{lem:minimal_exists} to the closed $G$-invariant space $Y \smallsetminus U$, we see that $U = Y$.  We now see by compactness that $Y$ is a finite union of $X$-interiors of closed $G$-minimal subspaces; moreover, the intersection between any two distinct closed $G$-minimal subspaces is empty by minimality, so actually $Y$ is a disjoint union of closed $G$-minimal subspaces and each of these is clopen.  Thus (iii) holds, completing the cycle of implications.
\end{proof}

A natural generalization of minimal homeomorphisms of a compact Hausdorff space are pointwise almost periodic homeomorphisms, defined as follows.

\begin{defn}\label{def:pap}
A subset $S$ of a group $G$ is \defbold{left syndetic} (or \defbold{left cobounded}) if $G= FS$ for a finite subset $F$ of $G$; analogously, a subset $S$ of $\bN$ is \defbold{syndetic} (or \defbold{cobounded}) if there is some $k \in \bN$ such that for all $n \in \bN$, we have $n+i \in S$ for some $0 \le i \le k$.  Let $G$ be a group of homeomorphisms of a Hausdorff space $X$.  Given a point $x \in X$ and a neighbourhood $U$ of $x$, define the set of \defbold{return times} $R_G(x,U) := \{g \in G \mid gx \in U\}$; for a single homeomorphism $g$ we define $R_g(x,U) := \{n \in \bZ \mid g^nx \in U\}$.  We say $G$ is \defbold{almost periodic at $x$} if $R_G(x,U)$ is a left syndetic subset of $G$ for every neighbourhood $U$ of $x$.  We say $G$ is \defbold{pointwise almost periodic} (\pap) on $X$ if it is almost periodic at every point in $X$.  For a single homeomorphism $g$, we say $g$ is \pap on $X$ if $\grp{g}$ is \pap on $X$.
\end{defn}

\begin{lem}[See for instance {\cite[Lemma 3.5]{Reid}}]\label{lem:pap_char}
Let $G$ be a group of homeomorphisms of a locally compact Hausdorff space $X$.  Then $G$ is \pap on $X$ if and only if, for every $x \in X$, the orbit closure $\overline{Gx}$ is a compact minimal $G$-space.
\end{lem}

We note a consequence of having a minimal-periodic partition that will be useful later.

\begin{lem}\label{lem:minimal_union}
Let $X$ be a nonempty compact Hausdorff space and let $g \in \Homeo(X)$ admit a minimal-periodic partition.  Let $K$ be a closed subspace of $X$ such that $\grp{g}x \cap K \neq \emptyset$ for all $x \in X$.  Then
\[
X = \bigcup^k_{i=0}g^iK
\]
for some finite $k$.
\end{lem}

\begin{proof}
By hypothesis,
\[
X = \bigsqcup^l_{i=1}X_\per(i) \sqcup \bigsqcup^m_{i=1}X_\aper(i),
\]
where the subspaces $X_\per(i)$ and $X_\aper(i)$ are all clopen, $g$-orbits on $X_\per(i)$ have size $n_i$, and $g$ acts minimally on $X_\aper(i)$.  Let $k_\per = \max\{n_1,\dots,n_l\}$; then $X_\per(j) \subseteq \bigcup^{k_\per}_{i=0}g^iK$ for $1 \le j \le l$.  So by replacing $K$ with $\bigcup^{k_\per}_{i=0}g^iK$, we can ignore the periodic points of $g$ and assume that $g$ is aperiodic.

Let $K_i = K \cap X_\aper(i)$.  Then $K_i$ is closed and our hypothesis ensures that $X_\aper(i) = \bigcup_{j \in \bZ}g^jK_i$.  By the Baire Category Theorem it follows that $K_i$ has nonempty interior $K'_i$.  Since $g$ acts minimally on $X_\aper(i)$ we see that $X_\aper(i) = \bigcup_{j \in \bZ}g^jK'_i$, and then by compactness there is $k_i \in \bN$ such that $X_\aper(i) = \bigcup^{k_i}_{j=0}g^jK$.  Taking $k = \max\{k_1,\dots,k_m\}$ finishes the proof.
\end{proof}

\subsection{Induced transformations}

Given a \pap homeomorphism $h$, we can use return times to define the \defbold{induced transformation} on a clopen subspace.

\begin{defn}
Let $g$ be a \pap (for example, minimal) homeomorphism of a compact Hausdorff space $X$ and let $Y$ be a clopen subset of $X$.  Then for all $x \in Y$, by Lemma~\ref{lem:pap_char} there exists $n > 0$ such that $g^nx \in Y$.  Define the \defbold{induced transformation} $g_Y:Y \rightarrow Y$ by setting $g_Yx = g^tx$, where $t$ is the least element of $R_g(x,Y) \cap \bN_{>0}$, or in other words $g_Yx$ is the $\le_g$-next element of $Y$ after $x$.

Given homeomorphisms $h_i \in \Homeo(X_i)$, define the \defbold{join} $h_1 \sqcup h_2$ to be the homeomorphism $g$ of $X_1 \sqcup X_2$ given by $gx = h_ix$ for all $x \in X_i$.
\end{defn}

The \pap property ensures that induced transformations are well-defined and well-behaved.  We see in particular that all induced transformations of a \pap homeomorphism belong to its topological full group.  The following basic facts will be used without further comment.

\begin{lem}\label{lem:induced}
Let $h$ be a \pap homeomorphism of a compact Hausdorff space $X$ and let $Y$ be a nonempty clopen subset of $X$.
\begin{enumerate}[(i)]
\item Given $x \in Y$, then $\grp{h_Y}x = \grp{h}x \cap Y$ and $(h\inv)_Y = (h_Y)\inv$.
\item $h_Y$ is a homeomorphism of $Y$ and the join $h_Y \sqcup \id_{X \smallsetminus Y}$ belongs to $\Fgp{h}$.
\item $h_Y$ is \pap on $Y$.  If $h$ is minimal, then so is $h_Y$.
\end{enumerate}
\end{lem}

\begin{proof}
(i)
From the definition, it is clear that for all $x \in Y$ the forward orbit of $h_Y$ visits every point in $\{x,hx,h^2x,\dots\} \cap Y$ in $\le_h$-order, and similarly the forward orbit of $(h\inv)_Y$ visits every point in $\{x,h^{-1}x,h^{-2}x,\dots\} \cap Y$ in reverse $\le_h$-order.  The conclusions are then clear. 

(ii)   We have already seen that $h_Y$ is bijective, so the join $h_Y \sqcup \id_{X \smallsetminus Y}$ is bijective.  Define $f: X \rightarrow \bN$ by setting $f(x)$ to be the least positive integer such that $h^{f(x)}x \in Y$ if $x \in Y$, and $f(x)=0$ otherwise.  By the \pap property, there is a neighbourhood $U$ of $x$ and $n > 0$ such that $h^nU \subseteq Y$; this ensures that $f$ is well-defined and continuous over $Y$.  Thus $h_Y \sqcup \id_{X \smallsetminus Y} \in \Fgp{h}$; in particular, 
\[
h_Y \sqcup \id_{X \smallsetminus Y} \in \Homeo(X) \text{ and } h_Y \in \Homeo(Y).
\]

(iii) follows immediately from (i) and Lemma~\ref{lem:pap_char}.
\end{proof}

Where there is no danger of confusion, we will identify $h_Y$ with the join $h_Y \sqcup \id_{X \smallsetminus Y}$, so that we can regard $h_Y$ as an element of $\Fgp{h}$.

In this setting, conjugating the induced transformations of $h$ by an element that centralizes $h$ has a predictable effect.

\begin{lem}\label{lem:induced:conjugate}
Let $h$ be a \pap homeomorphism of a compact Hausdorff space $X$, let $Y$ be a clopen subset of $X$ and let $f$ be a homeomorphism of $X$ that commutes with $h$.  Then
\[
fh_{Y} = h_{fY}f.
\]
\end{lem}

\begin{proof}
Let $x \in X$.  If $x \in X \smallsetminus Y$, we see that $fh_{Y}x = h_{fY}fx = fx$.  If $x \in Y$, say that $s$ is the smallest positive integer such that $h^sx \in Y$.  Then $fh_{Y}x = fh^sx$.  At the same time, we see that $s$ is the smallest positive integer such that $fh^sx \in fY$; since $f$ and $h$ commute, this is the same as the smallest positive integer $s$ such that $h^sfx \in fY$.  Thus $h_{fY}fx = h^sfx = fh^sx$, completing the proof that $fh_{Y} = h_{fY}f$.
\end{proof}

Let $(X_1,g_1)$ and $(X_2,g_2)$ be compact minimal systems.  A \defbold{Kakutani equivalence} of $(g_1,g_2)$ is a homeomorphism $\kappa: Y_1 \rightarrow Y_2$, where $Y_i$ is a nonempty clopen subset of $X_i$, such that $\kappa \circ (g_1)_{Y_1} = (g_2)_{Y_2} \circ \kappa$.  We say that $g_1$ and $g_2$ are \defbold{Kakutani equivalent} if a Kakutani equivalence exists, and \defbold{flip Kakutani equivalent} if $g_1$ is Kakutani equivalent to $g_2$ or $g\inv_2$. We note that (flip) Kakutani equivalence is indeed an equivalence relation.

\begin{lem}\label{lem:Kakutani_equivalence}
Kakutani equivalence and flip Kakutani equivalence are equivalence relations.
\end{lem}

\begin{proof}
It is clear that Kakutani equivalence is reflexive and symmetric.  To show transitivity, let $(X_i,g_i)$ be a compact minimal system for $1 \le i \le 3$ and suppose we have Kakutani equivalences $\kappa_{12}: Y_1 \rightarrow Y_2$ and $\kappa_{23}: Z_2 \rightarrow Z_3$, where $Y_i$ and $Z_i$ are nonempty clopen subspaces of $X_i$.  Since $g_2$ is minimal, there is some nonempty clopen subset $W$ of $Y_2$ and $t \ge 0$ such that $g^t_2W \subseteq Z_2$.  Note that the restriction of $\kappa_{12}$ to a homeomorphism from $Y'_1 = \kappa\inv(W)$ to $Y'_2 = W$ is also a Kakutani equivalence of $(g_1,g_2)$, since $((g_i)_{Y_i})_{Y'_i} = (g_i)_{Y'_i}$.  Thus we may assume without loss of generality that $Y_2 = W$.  Similarly, we may assume without loss of generality that $Z_2 = g^t_2Y_2$.  In this case, by Lemma~\ref{lem:induced:conjugate} we see that $g^t_2(g_2)_{Y_2} = (g_2)_{Z_2}g^t_2$.

We now obtain a homeomorphism $\kappa_{13}$ from $Y_1$ to $Z_3$ by setting $\kappa_{13}x = \kappa_{23}g^t_2\kappa_{12}x$.

Let $x \in Y_1$.  Then
\begin{align*}
\kappa_{13}(g_1)_{Y_1}x &=  \kappa_{23}g^t_2\kappa_{12}(g_1)_{Y_1}x =  \kappa_{23}g^t_2(g_2)_{Y_2}\kappa_{12}x\\
 &= \kappa_{23}(g_2)_{Z_2}g^t_2\kappa_{12}x = (g_3)_{Z_3}\kappa_{23}g^t_2\kappa_{12}x = (g_3)_{Z_3}\kappa_{13}x.
\end{align*}
Thus $\kappa_{13}$ is a Kakutani equivalence of $(g_1,g_3)$.  This proves that Kakutani equivalence is transitive and hence it is an equivalence relation.  It is then clear that flip Kakutani equivalance is also an equivalence relation.
\end{proof}

\section{The sign and minimal-periodic partitions}\label{sec:sign}

Given a group $G$ of homeomorphisms of a compact Hausdorff space $X$, recall that $X^G_{\per}$ is the union of all finite $G$-orbits on $X$ and $X^G_{\aper}$ is the union of all infinite $G$-orbits on $X$ (with the superscript omitted if $G$ is clear from context).

\subsection{Cyclic subgroups}\label{sec:cyclic-sign}

The following lemma is easily seen, for instance by observing that for any constant $c$, then $f(n) - c$ can only change sign finitely many times as $n$ increases.

\begin{lem}\label{lem:Lipschitz}
Let $f: \bN \rightarrow \bZ$ be an injective Lipschitz function (that is, $|f(n+1)-f(n)|$ is bounded over $n \in \bN$).  Then as $n$ increases, $f(n)$ tends to $+\infty$ or $-\infty$.
\end{lem}

Given an aperiodic homeomorphism $g$ and $h \in \Fgp{g}$, we can define the orientation of $h$-orbits on $X^{\grp{h}}_\aper$ relative to the action of $g$.

\begin{lem}\label{lem:aperiodic_drift}
Let $X$ be a compact Hausdorff space, let $g \in \Homeo(X)$ be aperiodic and let $h \in \Fgp{g}$.  Given $x \in X$ define $\phi_x: \bZ \rightarrow \bZ$ via $\phi_x(n) := c_{g,h^n}(x)$, and set
\begin{align*}
X_{h+} &:= \{x \in X \mid \phi_x(n) \rightarrow +\infty \text{ as } n \rightarrow +\infty\}; \\
X_{h-} &:= \{x \in X \mid \phi_x(n) \rightarrow -\infty \text{ as } n \rightarrow +\infty\}.
\end{align*}
Then the following holds.
\begin{enumerate}[(i)]
\item We have $X^{\grp{h}}_{\aper} = X_{h+} \sqcup X_{h-}$.
\item The sets $X_{h+}$ and $X_{h-}$ are each $\grp{h}$-invariant $F_\sigma$ subsets of $X$.
\item Every closed $\grp{h}$-minimal subset of $X^{\grp{h}}_{\aper}$ is contained in $X_{h+}$ or $X_{h-}$.
\item Suppose the sets $X_{h+}$, $X_{h-}$, $X_{h\inv+}$ and $X_{h\inv-}$ are closed.  Then $X_{h\inv +} = X_{h-}$ and $X_{h\inv-} = X_{h+}$, so $h$ is positive on $X_{h+}$ and negative on $X_{h-}$.
\end{enumerate}
\end{lem}

\begin{proof}
(i)
Since $h$ can only act as $g^i$ where $|i| \le |h|_g$, we see that $\phi_x$ is a Lipschitz map for each $x \in X$; hence by Lemma~\ref{lem:Lipschitz}, $X^{\grp{h}}_\aper = X_{h+} \sqcup X_{h-}$.
 
(ii)
The functions $\phi_x$ satisfy the formula
\[
\phi_{hx}(n) = \phi_{x}(n+1) - \phi_x(1),
\]
so the asymptotic behaviour of $\phi_{hx}$ is the same as that of $\phi_x$; it follows that $hX_{h+} = X_{h+}$ and $hX_{h-} = X_{h-}$.   We see that $X_{h+}$ and $X_{h-}$ are $F_\sigma$-sets from the arrangement of quantifiers in the definitions.

(iii)
Let $K$ be a closed $\grp{h}$-minimal subspace of $X^{\grp{h}}_{\aper}$.  Write $K_{h\pm} = K \cap X_{h\pm}$; then by (ii), $K$ is the disjoint union of $K_{h+}$ and $K_{h-}$, both of which are $\grp{h}$-invariant $F_{\sigma}$ sets.  By the Baire Category Theorem, one of $K_{h+}$ and $K_{h-}$, say $K_{h+}$, has nonempty interior in $K$; by minimality of $K$ it follows that $K = K_{h+}$.

(iv)
It is enough to show that $X_{h+} \cap X_{h\inv+}$ and $X_{h-} \cap X_{h\inv-}$ are empty; since they are compact, it is enough to show they contain no $\grp{h}$-minimal closed subspace.  So suppose for a contradiction that $X_{h+} \cap X_{h\inv+}$ contains the $\grp{h}$-minimal subspace $K$, and define $L \subseteq K$ by setting $x \in L$ if $x$ is the $\le_g$-least element of $\grp{h}x$, that is, $x \in Y$ and $c_{g,h^n}(x) \ge 0$ for all $n \in \bZ$.  Then $L$ is a closed subset of $K$, and from the definitions of $X_{h+}$ and $X_{h\inv+}$, we see that $L$ intersects each $h$-orbit on $K$ in exactly one point; Lemma~\ref{lem:minimal_union} then implies that every $h$-orbit on $K$ is finite, which is absurd.  From this contradiction we conclude that $X_{h+} \cap X_{h\inv+}$ is empty; the proof that $X_{h-} \cap X_{h\inv -}$ is empty is similar.
\end{proof}

The next lemma answers Question~\ref{que:aperiodic} in the case that $g$ is minimal.

\begin{lem}\label{lem:open_minimal}
Let $(X,g)$ be a compact minimal system and let $h \in \Fgp{g}$.  Define $X_{h+}$ and $X_{h-}$ as in Lemma~\ref{lem:aperiodic_drift}.  Then $X_{h+}$ and $X_{h-}$ are each the disjoint union of finitely many clopen $\grp{h}$-minimal subspaces of $X$.
\end{lem}

\begin{proof}
Given Lemma~\ref{lem:aperiodic_drift}(iii), it is enough to show that $X^{\grp{h}}_{\aper}$ is the disjoint union of finitely many $\grp{h}$-minimal clopen subspaces of $X$.  In turn, by Proposition~\ref{prop:locally_minimal}, we need only show that every closed $\grp{h}$-minimal subspace $K$ of $X^{\grp{h}}_{\aper}$ has nonempty interior; without loss of generality we may assume $K \subseteq X_{h+}$.

Given $x \in K$, there exists $n > 0$ such that $c_{g,h^n}(x) > 0$; furthermore, taking the least such $n$, we ensure that $c_{g,h^n}(x) \le |h|_g$.  Thus $h^nx = g^{t}x$ for some $0< t \le |h|_g$.  We have $h^nK = K$, in particular, $g^{t}x = h^nx \in K$, so $x \in g^{-t}K$.  Letting $x$ vary over $K$, we see that $K \subseteq \bigcup^{|h|_g}_{t=1}g^{-t}K$, and hence $g^{|h|_g}K \subseteq  M$, where $M = \bigcup^{|h|_g-1}_{t=0}g^{t}K$.  We observe that $gM \subseteq M$.  Now $M$ is closed and nonempty by construction, so by Lemma~\ref{lem:minimal_forwards}, we have $M = X$.  Thus $X$ is a union of finitely many $\langle g \rangle$-translates of $K$.  In particular, $K$ has nonempty interior as required.
\end{proof}

In particular, we deduce the sign partition for $h \in \Fgp{g}$.

\begin{proof}[Proof of Theorem~\ref{thm:sign}]
Set $h_\per$ to be the restriction of $h$ to $X_\per$.  By Lemma~\ref{lem:open_minimal}, $X_+ := X_{h+}$ and $X_- := X_{h-}$ are both clopen, so $X_\per$ is also clopen and hence by Lemma~\ref{lem:periodic}, $h_\per$ is strongly periodic.  In addition, by Lemma~\ref{lem:aperiodic_drift}(iv) we see that $h$ is positive on $X_+$ and negative on $X_-$.   We thus have a clopen partition of $X$ into $\grp{h}$-invariant pieces of the required form.
\end{proof}

\subsection{The sign and minimal-periodic partitions for virtually polycyclic subgroups}\label{sec:minimal-periodic}

We now extend the minimal-periodic partition from cyclic to virtually polycyclic subgroups of $\Fgp{g}$.

Given $x \in X$ and a finitely generated subgroup $G$ of $\Fgp{g}$, there are only finitely many infinite $G$-orbits on $\grp{g}x$.

\begin{lem}\label{lem:infinite_orbit_bound}
Let $X$ be a compact Hausdorff space, let $g \in \Homeo(X)$, let $S$ be a finite subset of $\Fgp{g}$, such that $S = S\inv$, and let $G = \grp{S}$.  Then there are at most $2|S|_g$ infinite $G$-orbits on $\grp{g}x$, where $|S|_g := \max \{ |c_{g,h}(x)| \mid x \in X, h \in S\}$.
\end{lem}

\begin{proof}
Let $s = |S|_g$.  Given an infinite $G$-orbit $\Omega$, then $\Omega$ is unbounded above or below (or both) with respect to $\le_g$.  We claim that there are at most $s$ infinite $G$-orbits unbounded above, and at most $s$ infinite $G$-orbits unbounded below; as the arguments are the same, we will only consider $G$-orbits that are unbounded above.  Suppose for a contradiction that we have distinct, and hence disjoint, $G$-orbits $\Omega_0,\Omega_1,\dots,\Omega_{s}$ on $\grp{g}x$, each unbounded above; for each $i$ choose $t_i \in \bZ$ such that $g^{t_i}x \in \Omega_i$; and let $t = \max\{t_0,t_1,\dots,t_{s}\}$.  Then we see that there is some element $h_i$ of $G$ such that $g^tx \le_g h_ig^{t_i}x$; by taking $h_i$ to be of minimal word length with respect to $S$, we also ensure that $h_ig^{t_i}x \le_g g^{t+s-1}x$, and clearly $h_ig^{t_i}x \neq h_jg^{t_j}x$ whenever $i \neq j$.  We have thus defined an injective map from $\{\Omega_0,\Omega_1,\dots,\Omega_s\}$ to the set $\{g^ts,g^{t+1}s,\dots,g^{t+s-1}x\}$, which is clearly impossible.
\end{proof}

In a virtually polycyclic group $G$, every subgroup $H$ of $G$ is itself virtually polycyclic, and hence finitely generated.  Here is another easy fact about virtually polycyclic groups.

\begin{lem}\label{lem:polycyclic_normal}
Let $G$ be an infinite virtually polycyclic group.  Then $G$ has an infinite torsion-free abelian characteristic subgroup.
\end{lem}

\begin{proof}
Let $H$ be a polycyclic characteristic subgroup of $G$ of finite index and let $N$ be the last infinite term in the derived series of $H$.  Then $N$ is finitely generated with finite derived subgroup, so its centre has finite index.  Thus for $n$ large enough, $N^{n!}$ is infinite and torsion-free abelian; by construction $N^{n!}$ is characteristic in $G$.
\end{proof}

\begin{lem}\label{lem:aperiodic_cyclic}
Let $(X,g)$ be a compact minimal system, let $G$ be a virtually polycyclic subgroup of $\Fgp{g}$ and let $x \in X$ be such that $Gx$ is infinite.  Let $K$ be the kernel of the action of $G$ on $Gx$ and let $N/K$ be an infinite torsion-free abelian normal subgroup of $G/K$ (as given by Lemma~\ref{lem:polycyclic_normal}).  Then given $x \in X$, there are $h_1,\dots,h_n \in N$ and $t_1,\dots,t_n \in \bZ$ such that
\[
Gx = \bigsqcup^n_{i=1} \grp{h_i}g^{t_i}x,
\]
where each of the sets $\grp{h_i}g^{t_i}x$ is infinite.
\end{lem}

\begin{proof}
Since $N/K$ is infinite and torsion-free abelian, we can take an infinite cyclic subgroup $\grp{h}$ of $N$ that acts faithfully on $Gx$.  By Theorem~\ref{thm:sign}, the finite orbits of $\grp{h}$ are of bounded length, so $\grp{h}$ must have an infinite orbit $\grp{h}y$ where $y = g^tx \in Gx$.  Since $\grp{h}K$ is normal in $N$, every $h$-orbit on $Ny$ is infinite, and hence by Lemma~\ref{lem:infinite_orbit_bound}, $Ny$ is partitioned into finitely many $h$-orbits.  In turn, since $N$ is normal in $G$, we see that $Gx$ is partitioned into infinite $N$-orbits; by Lemma~\ref{lem:infinite_orbit_bound} again, there are only finitely many of these.  Thus $Gx$ is expressible as a disjoint union of finitely many sets, such that each set is an infinite orbit of a cyclic subgroup of $N$, as required.
\end{proof}

We now complete the proof of Theorem~\ref{thm:minimal_periodic} and Corollary~\ref{cor:pap}.

\begin{proof}[Proof of Theorem~\ref{thm:minimal_periodic}]
Let $g$ be a homeomorphism of a compact infinite Hausdorff space $X$ such that $g$ admits a minimal-periodic partition, and let $G \le \Fgp{g}$ be virtually polycyclic.  It is enough to show $G$ has a minimal-periodic partition with the required properties on $X^{\grp{g}}_{\per}$ and on each infinite clopen $\grp{g}$-minimal subspace of $X$.  On $X^{\grp{g}}_{\per}$, the conclusion follows easily from Lemma~\ref{lem:periodic} and compactness.  So we may assume without loss of generality that $g$ is minimal on $X$.

It is clear from Lemma~\ref{lem:aperiodic_cyclic} that $X^G_\aper = \bigcup_{h \in G}X^{\grp{h}}_{\aper}$; by Theorem~\ref{thm:sign} and Lemma~\ref{lem:periodic}, it follows that $X^G_{\aper}$ is clopen.  By compactness, indeed
\[
X^G_{\aper} = \bigcup^k_{i=1} X^{\grp{h_i}}_{\aper}
\]
for some $h_1,\dots,h_k \in G$.  We then deduce from Lemma~\ref{lem:open_minimal} that $X^G_{\aper}$ is partitioned into finitely many clopen $G$-minimal subspaces $X_{\aper}(1),\dots,X_{\aper}(m)$.

We now know that $X_{\per}$ is compact.  Applying Lemma~\ref{lem:periodic} to $X_{\per}$, we see that $X_{\per}$ is partitioned into finitely many clopen subspaces such that the point stabilizer in $G$ is constant on each subspace, and hence finitely many clopen $G$-invariant subspaces such that the point stabilizer in $G$ belongs to a constant conjugacy class on each subspace.  We have now established that $G$ admits a minimal-periodic partition on $X$.

Now consider an infinite closed $G$-minimal subspace $Y$ of $X$, and let $x \in Y$.  We see that the kernel $K$ of the action of $G$ on $Y = \ol{Gx}$ is the same as the kernel of the action of $G$ on $Gx$.  Taking $N$ an infinite torsion-free abelian normal subgroup of $G/K$, we see by Lemma~\ref{lem:aperiodic_cyclic} and its proof that $Gx$ is partitioned into finitely many infinite $N$-orbits, such that on each $N$-orbit, $N$ acts as a virtually cyclic group.  After replacing $N$ with $\grp{h^t \mid h \in N}K$ for a suitable natural number $t$, we can in fact ensure that $Gx$ is partitioned into finitely many infinite $N$-orbits $Ny_1,\dots,Ny_r$ (possibly more such orbits than before), such that now $N/N_i \cong \bZ$, where $N_i$ is the kernel of the action of $N$ on $Ny_i$.  Note that $N_i$ is also the kernel of the action of $N$ on $\ol{Ny_i}$.  Now
\[
Y = \ol{Gx} = \bigcup^r_{i=1}\ol{Ny_i};
\]
since $N$ itself has a minimal-periodic partition, we see that $Y \subseteq X^N_{\aper}$ and each of the subspaces $\ol{Ny_i}$ is one of the infinite minimal clopen $N$-invariant subspaces of $X$, so if $\ol{Ny_i}$ and $\ol{Ny_j}$ are distinct, they are disjoint.  After reordering, the sequence $\ol{Ny_1},\dots,\ol{Ny_n}$ lists all the distinct subspaces of $X$ in $\ol{Ny_1},\dots,\ol{Ny_r}$, and we have
\[
Y = \ol{Ny_1} \sqcup \ol{Ny_2} \sqcup \dots \sqcup \ol{Ny_n},
\]
giving the required decomposition of $Y$.
\end{proof}

\begin{proof}[Proof of Corollary~\ref{cor:pap}]
Let $f$ be a \pap homeomorphism of $X$, let $G$ be a virtually polycyclic subgroup of $\Fgp{f}$, let $x \in X$ and let $Y = \overline{Gx}$.  By Lemma~\ref{lem:pap_char} it is enough to show $Y$ is compact and $G$-minimal.  We may assume $Gx$ is infinite, otherwise there is nothing to prove.  Then $Y$ is contained in some orbit closure $Z$ of $f$; since $f$ is \pap on $X$, $Z$ is compact and $\grp{f}$-minimal.  Now by Theorem~\ref{thm:minimal_periodic}, $x$ belongs to some clopen $G$-minimal subspace $Y'$ of $Z$; clearly we then have $Y = Y'$.
\end{proof}

\section{The relationship between positive and strongly positive elements}\label{sec:positive}

We note some easy facts about $\Fgpp{g}$ when $g$ is an aperiodic homeomorphism.

\begin{prop}\label{prop:po:cocycle}
Let $X$ be an infinite compact Hausdorff space and let $g \in \Homeo(X)$ be aperiodic.
\begin{enumerate}[(i)]
\item The set $\Fgpp{g}$ is a submonoid of $\Fgp{g}$ such that $\Fgp{g} = \Fgpp{g}\grp{g}$ and $g\Fgpp{g}g\inv = \Fgpp{g}$.
\item The group $\Fgp{g}$ is the universal enveloping group of $\Fgpp{g}$, that is, the embedding $\iota: \Fgpp{g} \rightarrow \Fgp{g}$ has the property that every homomorphism from $\Fgpp{g}$ to a group uniquely factors through $\iota$.
\item Let $a,b \in \Fgp{g}$.  Then $ba\inv \in \Fgpp{g}$ if and only $c_{g,a}(x) \le c_{g,b}(x)$ for all $x \in X$.
\item Suppose $g$ is \pap on $X$ and let $h \in \Fgpp{g}$.  Then $hg\inv_{\supp(h)} \in \Fgpp{g}$.
\end{enumerate}
\end{prop}

\begin{proof}
Write $G = \Fgp{g}$ and $M = \Fgpp{g}$.  Note that given $h \in G$, we have $h \in M$ if and only if $\min_{x \in X}c_{g,h}(x) \ge 0$.

(i)
Clearly $1 \in M$; given $a,b \in M$ and $x \in X$ then
\[
c_{g,ab}(x) = c_{g,a}(bx) + c_{g,b}(x) \ge 0,
\]
so $ab \in M$.  Thus $M$ is a submonoid.  Given $h \in G$, we see that $hg^{-r} \in M$ where $r = \min_{x \in X}c_{g,h}(x)$, so $h \in \grp{g}M$.  We see that $gMg\inv = M$, or in other words $gM = Mg$, using the identity
\[
\forall x \in X, m \in M: \; c_{g,gmg\inv}(x) = c_{g,m}(g\inv x).
\]

(ii)
Writing $E$ for the universal enveloping group of $M$, with homomorphism $\iota:M \rightarrow E$, then $\iota$ is injective (since $M$ embeds in a group) and $E$ is generated by $\iota(M)$ as a group.  We then observe that if a product of elements of $M \cup \{g\inv\}$ reduces to the trivial word in $G$, then the same is true of the corresponding word $w$ expressed as a product of elements of $\iota(M) \cup \{\iota(g)\inv\}$ in $E$: First multiply $w$ on the right by $\iota(g^k)\iota(g)^{-k}$, where $k$ is the number of occurrences of $\iota(g)\inv$ in $w$; use multiplication inside $\iota(M)$ to move positive powers of $\iota(g)$ to the left; and delete occurrences of $\iota(g)\inv \iota(g)$ when they occur.  The result is an expression equivalent to $w$ of the form $\iota(m)\iota(g)^{-k}$ for some $m \in M$, at which point we see that we must have $m = g^k$ in $G$ and the expression cancels out.  In particular, $\iota(M) \cup \{\iota(g)\inv\}$ generates $E$ as a monoid and we see that the canonical homomorphism $E \rightarrow G$ is an isomorphism.

(iii)
Given $a,b \in G$ and $x \in X$, we see that
\[
c_{g,ba\inv}(x) = c_{g,b}(a\inv x) + c_{g,a\inv}(x) = c_{g,b}(a\inv x) - c_{g,a}(a\inv x).
\]
Thus $ba\inv \in M$ if and only if $c_{g,a}(a\inv x) \le c_{g,b}(a\inv x)$ for all $x \in X$, or in other words, $c_{g,a}(x) \le c_{g,b}(x)$ for all $x \in X$.

(iv)
It is clear from how induced transformations are defined that $c_{g,g_{\supp(h)}}(x) \le c_{g,h}(x)$ for all $x \in X$; the conclusion then follows from (iii).
\end{proof}

We now prove Theorem~\ref{thm:positive_conjugation}, starting with two lemmas.

\begin{lem}\label{lem:cocycle_comparison}
Let $X$ be a compact Hausdorff space and let $g$ and $h$ be permutations of $X$ such that $\grp{g}x = \grp{h}x$ and $|\grp{g}x| = \infty$ for all $x \in X$.  Suppose that one of $g$ and $h$ is a homeomorphism, and that there is a continuous function $c_{g,h}: X \rightarrow \bZ$ such that for all $x \in X$, $hx = g^{c_{g,h}(x)}x$.  Then $g$ and $h$ are both homeomorphisms and $\Fgp{g} = \Fgp{h}$.
\end{lem}

\begin{proof}
If $g$ is a homeomorphism, then clearly $h$ is a homeomorphism, since $h$ acts locally by powers of $g$.  Conversely, suppose $h$ is a homeomorphism and let $x \in X$.  There is some $i \in \bZ$ such that $h^ix = gx$, and then a neighbourhood $U$ of $x$ on which $h^i$ acts as a constant power of $g$; since $|\grp{g}y| = \infty$ for all $y \in X$, no nonzero power of $g$ has a fixed point, so we must have $h^iy = gy$ for all $y \in U$.  Thus there is a well-defined locally constant cocycle $c_{h,g}: X \rightarrow \bZ$, showing that $g$ is a homeomorphism.  So if either of $g$ and $h$ are homeomorphisms, then both are; since the cocycles $c_{g,h}$ and $c_{h,g}$ both exist, we have $\Fgp{g} = \Fgp{h}$.
\end{proof}

\begin{lem}\label{lem:delta}
Let $(X,g)$ be a compact minimal system and let $h \in \Fgpnn{g}$.  Set $h'$ to be the unique permutation of $X$ such that $h'x \ge_g x$ for all $x \in X$ and $h'$ has the same orbits as $h$.  If $x$ is a fixed point of $h$, set $\delta(x,t)=0$ for all $t \in \bZ$; otherwise set $\delta(x,t) = c_{h',h^t}(x) - t$.
\begin{enumerate}[(i)]
\item We have $\Fgp{h} = \Fgp{h'}$; in particular, $h'$ is a homeomorphism.
\item $\delta$ is a continuous function from $X \times \bZ$ to $\bZ$ that takes only finitely many values.
\item Define
\[
Z := \{x \in X \mid \forall t \in \bZ: \delta(x,t) \ge 0\}.
\]
Then for all $(x,s) \in X \times \bZ$,
\[
h^sx \in Z \Leftrightarrow s \in \underset{t \in \bZ}{\arg\min} \; \delta(x,t).
\]
In particular, $Z$ is clopen and intersects every $h$-orbit.  In addition, $h_{Z} = (h')_{Z}$.
\end{enumerate}
\end{lem}

\begin{proof}
Let $X_\aper$ be the set of aperiodic points for $h$.  We first note that $h'$ clearly exists (given the fact that $h$ has no nontrivial finite orbits), and that $c_{h',h^t}$ is well-defined: for all $x \in X_\aper$ and $t \in \bZ$, we see that $x$ and $h^tx$ belong to the same infinite $h'$-orbit, and hence $h^tx = (h')^{t'}x$ for some unique $t' \in \bZ$.  Moreover, the $\le_{h'}$-ordering on the points between $x$ and $hx$ in the $h$-orbit is the same as the $\le_g$-ordering; thus
\[
\sup_{x \in X} |c_{h',h}(x)| \le |h|_g.
\]
We deduce that $\delta(x,t)$ has linearly bounded dependence on $t$: specifically, for any $x \in X_\aper$ and $t \in \bZ$,
\[
|\delta(x,t+1) - \delta(x,t)| = |c_{h',h}(h^tx) - 1| \le |h|_{g} + 1.
\]
For the remainder of the argument we focus on the clopen set $X_\aper$ of aperiodic points for $h$, which are also aperiodic for $h'$, since by definition $\delta(x,t)=0$ whenever $x \in X \smallsetminus X_\aper$.

Let
\[
Y_+ := \{y \in X_\aper \mid \forall n \ge 0: h^ny \ge_g y\}; \quad Y_- := \{y \in X_\aper \mid \forall n \ge 0: h^{-n}y \le_g y\}.
\]
We see that $Y_+$ is a closed set, since it is defined by a conjunction of cocycle values.  Since $h \in \Fgpnn{g}$, we see that every forward $h$-orbit has a $\le_g$-least point, since $c_{g,h^n}(x) \rightarrow +\infty$ as $n \rightarrow +\infty$; if $y$ is the $\le_g$-least point in $\{h^nx \mid n \ge 0\}$, then $y \in Y_+$.  By a similar argument, one sees that $Y_-$ is a closed set that intersects every $h$-orbit.  Thus by Lemma~\ref{lem:minimal_union}, there is $k_0 \in \bN$ such that
\[
X_\aper = \bigcup^{k_0}_{i=0}h^{-i}Y_+ = \bigcup^{k_0}_{i=0}h^{-i}Y_-.
\]

\emph{Claim: Let $x \in X_\aper$.  Suppose $y \in \grp{h}x \cap Y_-$ and $z \in \grp{h}x \cap Y_+$ are such that $y \le_g x \le_g z$.  Then $y \le_h x \le_h z$.}

From the definition of $Y_-$, we see that we cannot have $y >_h x$, as this would imply $y >_g x$.  Thus, since $y \in \grp{h}x$, we must have $y \le_h x$.  A similar argument shows that $x \le_h z$.  This proves the claim.

Let $x \in X$ and choose $y \in Y_-$ such that $x \le_h y$.  We can choose $y$ to be $h^k x$ for $0 \le k \le k_0$.  From the definition of $Y_-$ we see that $x \le_g y$, so $y = (h')^{k'}x$ for some $0 \le k' \le |h|_gk$.  Now consider $l \ge 0$ such that $h^ly \in Y_+$ and $h^ly \ge_g y$.  Let $A = \{z \in \grp{h}x \mid y <_g z \le_g h^ly\}$ and let $B = \{z \in \grp{h}x \mid y <_h z \le_h h^ly\}$.  Then $A \subseteq B$ by the claim.  Clearly $|B| = l$, so $|A| \le l$.  Meanwhile, since $h'$ is strongly positive and $\grp{h'}$ has the same orbits as $\grp{h}$, we see that $h^ly = (h')^{|A|}y$.  Hence
\[
(h')^{k'+l}x = (h')^ly \ge_{h'} h^ly = h^{k+l}x.
\]
In particular,
\[
\delta(x,t) = c_{h',h^t}(x) - t \le 2|h|_gk^2_0,
\]
where $t = k+l$.  There are syndetically many values $t \in \bN$ satisfying this inequality; since $\delta(x,t)$ has linearly bounded dependence on $t$, it follows that $\delta(x,t)$ is bounded above, uniformly over all of $x \in X$ and $t \ge 0$.  Similarly, we can take $y' \in \grp{h}x \cap Y_+$ such that $x \le_g y'$,  and $l \ge 0$ such that $h^ly' \in Y_-$.  Let $A' = \{z \in \grp{h}x \mid y' <_g z \le_g h^ly'\}$ and let $B' = \{z \in \grp{h}x \mid y' <_h z \le_h h^ly'\}$.  This time $B' \subseteq A'$, and a similar argument to before shows that $\delta(x,t)$ is bounded below, uniformly over all of $x \in X$ and $t \in \bN$.  We can extend this to a bound over all $t \in \bZ$ by considering the actions of $h\inv$ and $(h')\inv$ in the same way.

Note that $c_{h',h}(x)$ counts the number of points $y \in \grp{h}x$ such that $x <_g y \le_g hx$ (if $c_{g,h}(x) \ge 0$) or $hx \le_g y <_g x$ (if $c_{g,h}(x) < 0$).  The fact that $\delta$ is bounded means that the number of such points is determined by the values of $c_{g,h^i}(x)$ where $i$ ranges over a finite set of values independently of $x$.  Hence $c_{h',h}$ and $\delta$ are both continuous.  In particular, we conclude by Lemma~\ref{lem:cocycle_comparison} that $h'$ is a homeomorphism and $\Fgp{h} = \Fgp{h'}$.  This completes the proof of (i) and (ii).

The formula $c_{h',ab}(x) = c_{h',a}(bx) + c_{h',b}(x)$ implies that $\delta(x,t+s) = \delta(h^sx,t) + \delta(x,s)$ for all $x \in X$ and $s,t \in \bZ$, or in other words
\[
\delta(h^sx,t) = \delta(x,t+s) - \delta(x,s).
\]
We immediately deduce the claimed characterization in (iii) of when $h^sx \in Z$.  The fact that $Z$ is clopen and intersects every $h$-orbit then follows immediately from (ii).

We see that on $Z$, the action of $h_{Z}$ is strongly positive with respect to $h'$, but also has the same orbits as $(h')_{Z}$.  The only way this can happen is if $h_{Z} = (h')_{Z}$.  Thus (iii) is proved.
\end{proof}

\begin{proof}[Proof of Theorem~\ref{thm:positive_conjugation}]
Let $(X,g)$ be a compact minimal system.

(i)
Let $h \in \N_{\Homeo(X)}(\Fgp{g})$.  Then $\Fgp{g} = \Fgp{hgh\inv}$.  Moreover, since $hgh\inv$ acts minimally, we see by Theorem~\ref{thm:sign} that either $hgh\inv \in \Fgpnn{g}$ or $(hgh\inv)\inv \in \Fgpnn{g}$, but not both.

Suppose that $(hgh\inv)\inv \in \Fgpnn{g}$ and let $x \in X$.  Then $c_{g,(hgh\inv)^n}(x)$ tends to $-\infty$ as $n$ tends to $+\infty$, and vice versa.  It follows that for a fixed $a \in \Fgp{g}$, then $c_{g,a^n}(x)$ tends to $-\infty$ as $n$ tends to $+\infty$ if and only if $c_{hgh\inv,a^n}(x)$ tends to $+\infty$ as $n$ tends to $+\infty$, and vice versa.  Thus 
\[
h\Fgpnn{g}h\inv = \Fgpnn{hgh\inv} = (\Fgpnn{g})\inv.
\]
Similarly, if $hgh\inv \in \Fgpnn{g}$, we see that $h\Fgpnn{g}h\inv = \Fgpnn{g}$.  We conclude that in any case, $h\Fgpnn{g}h\inv \in \{\Fgpnn{g}, (\Fgpnn{g})\inv\}$.

Now suppose $h \in \Fgp{g}$.  Then we can bound $c_{g,(hgh\inv)^n}(x)$ below as follows:
\[
c_{g,(hgh\inv)^n}(x) = c_{g,hg^nh\inv}(x) = c_{g,h}(g^nh\inv x) + c_{g,g^n}(h\inv x) + c_{g,h\inv}(x) \ge -2|h|_g + n.
\]
In particular, we see that $c_{g,hg^nh\inv}(x)$ tends to $+\infty$ as $n$ tends to $+\infty$, so $hgh\inv \in \Fgpnn{g}$, and hence $h\Fgpnn{g}h\inv = \Fgpnn{g}$.

(ii)
Let $h \in \Fgpnn{g}$ and define $h' = \pipp(h)$ as in Lemma~\ref{lem:delta}.  We can partition $X$ according to the minimal-periodic partition for $h$, which is the same as the minimal-periodic partition for $h'$, and conjugate independently on each part; thus we may assume $h$ and $h'$ are minimal.  We also see in this case that $\Fgp{h'} \cap \Fgpp{g} = \Fgpp{h'}$.  Thus without loss of generality we can assume $h' = g$.

Let $\delta$ and $Z$ be as in Lemma~\ref{lem:delta}; by Lemma~\ref{lem:delta}(iii) we have $g_{Z} = h_{Z}$.

Let $x \in Z$ and let $t = c_{h,h_{Z}}(x)$.  We see that $\delta(x,t) \ge 0$ and $-\delta(x,t) = \delta(h^tx,-t) \ge 0$, so $\delta(x,t)=0$.  In other words, $g^tx = h^tx$; since $g_{Z} = h_{Z}$ it follows that $c_{g,g_Z}(x) = t$.  Thus $R_g(x,Z) = R_h(x,Z)$ for all $x \in Z$.

We now define a map $k: X \rightarrow X$ by setting $kx = h^{s_x}g^{-s_x}x$, where $s_x \in \bZ$ is chosen so that $g^{-s_x}x \in Z$.  The properties of $Z$ we have established so far ensure that a suitable choice of $s_x$ exists and that all suitable choices will result in the same value of $kx$.  To see that $k$ is bijective, note that given $z \in Z$, then $\grp{h}z = \grp{g}z$, and given $t \in \bZ$ then $kg^tz = h^tz$; it is also easily seen that $k$ acts as a locally constant power of $g$, so $k \in \Fgp{g}$.  Given $x \in X$, then
\[
kgk\inv (kx) = kgx = h^{s_{x}+1}g^{-s_x-1}(gx) = h(h^{s_x}g^{-s_x}x) = hkx;
\]
thus $h = kgk\inv$.  We then see that
\[
c_{g,k}(x) = c_{g,h^{s_x}}(g^{-s_x}x) + c_{g,g^{-s_x}}(x) = c_{g,h^{s_x}}(g^{-s_x}x) - s_x = \delta(g^{-s_x}x,s_x) \ge 0.
\]
Thus $k \in \Fgpp{g}$.

To see that $h'$ is the only $\Fgp{h}$-conjugate of $h$ that is strongly positive with respect to $g$, note that any $\Fgp{h}$-conjugate $h''$ of $h$ would have the same orbits as $h$; if we also have $h'' \in \Fgpp{g}$, it is then clear that $h'' = h'$.

(iii)
Let $k \in \Fgpp{g} \smallsetminus \{1\}$; we must show there is $h \in \Fgpp{g}$ such that $k \not\in h\Fgpp{g}h\inv$, or equivalently, $h\inv kh \not\in \Fgpp{g}$.  Since $k$ is nontrivial and positive, by the minimal-periodic partition there is a clopen $\grp{k}$-minimal subspace $U$ of $X$.  There is then a clopen subset $V$ of $U$ such that $V$ and $kV$ are disjoint.  We then take the homeomorphism $h$ of $X$, where
\[
hx = 
\begin{cases}
kk_V x &\mbox{if} \;  x \in V\\
k_Vk\inv x &\mbox{if} \;  x \in kV \\
x &\mbox{otherwise}
\end{cases}
\]
Notice that $h \in \Fgpp{k}$ and hence $h \in \Fgpp{g}$.  Now consider $h\inv kh$ for $x \in kV$.  We have
\[
h\inv kh x =  (kk_V)\inv k k_Vk\inv x = k\inv x
\]
which is the image of $x$ under a negative power of $g$, so $h\inv kh \not\in \Fgpp{g}$.
\end{proof}

\begin{rmk}
The case $h\Fgpnn{g}h\inv = (\Fgpnn{g})\inv$ of Theorem~\ref{thm:positive_conjugation}(i) does occur.  For example, let $X = \{0,1\}^*$ be the set of infinite binary strings with the pointwise convergence topology and let $g$ act as $g(0w) = 1w$ and $g(1w) = 0g(w)$ for all strings $w$.  Then $(X,g)$ is a well-known Cantor minimal system, namely the binary odometer.  If we let $h$ be the homeomorphism acting on $X = \{0,1\}^*$ by changing each digit $0$ to $1$ and vice versa, one sees that $hgh\inv = g\inv$, and hence $h\Fgpnn{g}h\inv = (\Fgpnn{g})\inv$.

When $X$ is zero-dimensional, the normalizer of $\Fgp{g}$ in $\Homeo(X)$ is in fact the whole automorphism group of $\Fgp{g}$ as a group, by \cite[Theorem~2.14]{Rubin}.
\end{rmk}

\section{Orbit numbers and the index map}\label{sec:index}

Given a compact minimal system $(X,g)$, we introduce two functions $o^+$ and $o^-$ from $\Fgp{g}$ to $\bN$ that respectively count the positive and negative infinite orbits of an element along a $g$-orbit.  Their difference will turn out to be the well-known index map, as discussed in the introduction.

\begin{prop}\label{prop:orbit_number}
Let $(X,g)$ be a compact minimal system.  Then there are functions $o^+_g$ and $o^-_g$ from $\Fgp{g}$ to $\bN$ such that $h$ has exactly $o^+_g(h)$ positive orbits and $o^-_g(h)$ negative orbits on $\grp{g}x$ for all $x \in X$ and $h \in \Fgp{g}$.  Moreover, the following holds:
\begin{enumerate}[(i)]
\item We have $o^\pm_g(h) = o^\pm_g(khk\inv)$ for all $k \in \Fgp{g}$.
\item We have
\[
o^\pm_g(h) \ge o^\pm_g(h_A) = o^\pm_{g_A}(h_A)
\]
for all clopen subspaces $A$ of $X$, with $o^\pm_g(h) = o^\pm_g(h_A)$ if and only if $A$ intersects every infinite $h$-orbit of the corresponding orientation.
\end{enumerate}
\end{prop}

\begin{proof}
Let $h \in \Fgp{g}$, let $n = |h|_g$ and let $x \in X$.  We see that every infinite $h$-orbit on $\grp{g}x$ passes through the set $\{x,gx\dots,g^{n-1}x\}$, so the number $o^x_g(h)$ of infinite $h$-orbits on $\grp{g}x$ is at most $n$.  Note that $o^{kx}_g(h) = o^x_g(h)$ for all $k \in \Fgp{g}$, since $\grp{g}kx = \grp{g}x$.  Given $x \in X$ and $k \in \Fgp{g}$, then by Theorem~\ref{thm:positive_conjugation}(i), $kgk\inv$ has the same orbits as $g$ with the same orientation; thus
\[
o^x_g(khk\inv) = o^{kx}_{kgk\inv}(khk\inv) = o^x_g(h).
\]

Suppose for the moment that $h \in \Fgpp{g}$.  Choose a neighbourhood $U_0$ of $x$ small enough that the sets $U_i := g^iU_0$ are disjoint and $c_{g,h}$ is constant on $U_i$ for $0 \le i \le 2n$.  The fact that $c_{g,h}$ is constant on $U_i$ ensures that for $0 \le i < n$, we have $hU_i \in \{U_i \mid 0 \le i < 2n\}$, and we can calculate $o^x_g(h)$ as the number of connected components in the finite graph $\Gamma$ with vertices $\{U_i \mid 0 \le i < n\}$ and an edge between $U_i$ and $U_j$ if $hU_i = U_j$.  In particular, we see that $o^x_g(h) = o^y_g(h)$ for all $y \in U_0$.  Thus if we hold $g$ and $h$ constant, $o^x_g(h)$ is a continuous function of $x \in X$ that is constant on the dense subset $\grp{g}x$ of $X$.  It follows that $o^x_g(h)$ does not depend on the choice of $x \in X$.  Thus we can define $o^+_g(h) := o^x_g(h)$ for any $x \in X$ and $o^-_g(h)=0$.

Given a nonempty clopen subset $A$ of $X$, then $\grp{g}x \cap A$ is a $\grp{g_A}$-orbit, and $\grp{h_A}$ acts transitively on $\grp{h}y \cap A$ for every $y \in X$ such that $\grp{h}y \cap A$ is nonempty.  Given $y \in X$ such that $\grp{h}y$ is infinite, then $h$ acts minimally on $\overline{\grp{h}y}$, by Theorem~\ref{thm:minimal_periodic}; in particular, $\overline{\grp{h}y}$ is a perfect space, so $\grp{h}y \cap A$ is either empty or infinite.  Thus if $x \in Y$, then
\[
o^+_g(h) \ge o^+_g(h_A) = o^+_{g_A}(h_A)
\]
with $o^+_g(h) = o^+_g(h_A)$ if and only if $A$ intersects every infinite $h$-orbit on $\grp{g}x$.

This completes the proof of the proposition in the case that $h \in \Fgpp{g}$.  By Theorem~\ref{thm:positive_conjugation}(ii) and the invariance of $o^x_g(h)$ under conjugation of $h$, we immediately deduce the proposition for $h \in \Fgpnn{g}$.

For the general case, write $h_+$ and $h_-$ for the restrictions of $h$ to its positive and negative part respectively.  We now observe that the number of positive $h$-orbits on $\grp{g}x$ for $x \in X$ is exactly $o^+_{g}(h_+)$, and the number of negative $h$-orbits is exactly $o^+_{g\inv}(h_-)$.  Now define $o^+_g(h) := o^+_{g}(h_+)$ and $o^-_{g}(h) := o^+_{g\inv}(h_-)$.  Since $h_+ \in \Fgpnn{g}$ and $h_- \in \Fgpnn{g\inv}$, the desired conclusions all follow from the positive case of the proposition.
\end{proof}

With respect to a given compact minimal system $(X,g)$, and given $h \in \Fgp{g}$, we define the \defbold{positive orbit number} of $h$ to be $o^+_g(h)$, the \defbold{negative orbit number} to be $o^-_g(h)$, and the \defbold{orbit number} to be $o_g(h) := o^+_g(h) + o^-_g(h)$.

Recall now the index map, which is the unique homomorphism $I: \Fgp{g} \rightarrow \bZ$ such that $I(g)=1$.  Unsurprisingly given its uniqueness, it has several equivalent definitions; we recall here the definition of \v{S}t\v{e}p\'{a}nek--Rubin, stated more explicitly in the survey article \cite{Cor} of Y. Cornulier.

\begin{defn}
Let $X$ be a compact Hausdorff space, let $g$ be a homeomorphism of $X$ and let $x \in X$ be such that the $g$-orbit of $x$ is infinite.  Given $h \in \Homeo(X)$, write $h^{\bN}x$ for the set $\{h^nx \mid n \ge 0\}$.  Define a map $I^x_g: \Fgp{g} \rightarrow \bZ$ by setting
\[
I^x_g(h) = |h\inv g^{\bN}x \smallsetminus g^{\bN}x| - |g^{\bN}x \smallsetminus h\inv g^{\bN}x|.
\]
In words, $I^x_g(h)$ is the net transfer of points into the set $g^{\bN}x$ via the application of $h$.
\end{defn}

\begin{lem}[{\cite[Proposition~3.3.1]{Cor}}]\label{lem:Cor_index}
The map $I^x_g$ is a group homomorphism that does not depend on the choice of $x$ within its $g$-orbit; one has $I^x_g(g)=1$ (and hence $I^x_g$ is surjective).  Moreover, if $\mathrm{Hom}(\Fgp{g},\bZ)$ is equipped with the pointwise convergence topology, then the map $x \mapsto I^x_g$, defined on the union of the infinite orbits of $g$, is continuous.
\end{lem}

In particular, if $g$ acts minimally then $I^x_g = I^y_g$ for all $x,y \in X$, and so we drop the superscript and write $I_g$.  We will also drop the subscript and write $I$ for the index map when $g$ is clear from context.

We can easily express this version of the index map using the positive and negative orbit numbers.

\begin{prop}\label{prop:index_sum}
Let $(X,g)$ be a compact minimal system.  Then $I_g(h) = o^+_g(h) - o^-_g(h)$ for all $h \in \Fgp{g}$.
\end{prop}

\begin{proof}
Let $h \in \Fgp{g}$.  It is clear that the functions $o^\pm_g$ and $I_g$ are additive with respect to decompositions of the input $h$ into parts with disjoint clopen support and that $o^+_g(h) = o^-_g(h) = I_g(h) = 0$ if $h$ has finite order.  Thus by Theorem~\ref{thm:sign}, we reduce to the case that $h$ is positive or negative.  It is also clear that replacing $h$ with $h\inv$ reverses the roles of $o^+_g$ and $o^-_g$, and that we have $I_g(h\inv) = -I_g(h)$; so we only need to consider $h \in \Fgpnn{g}$.

By Proposition~\ref{prop:orbit_number}, $o^\pm_g$ is invariant under conjugation in $\Fgp{g}$; the same is true of $I_g$, since it is a homomorphism to an abelian group.  So we are free to replace $h$ with a conjugate of $h$.  By Theorem~\ref{thm:positive_conjugation}, we may therefore assume $h \in \Fgpp{g}$.

Given $h \in \Fgpp{g}$, for each infinite $h$-orbit $\Omega \subseteq \grp{g}x$, there is exactly one $y \in \Omega \smallsetminus g^{\bN}x$ such that $hy \in g^{\bN}x$, whereas for all $y \in g^{\bN}x$, we also have $hy \in g^{\bN}x$.  Thus 
\[
|h\inv g^{\bN}x \smallsetminus g^{\bN}x| = o^+_g(h) \quad \text{ and } \quad |g^{\bN}x \smallsetminus h\inv g^{\bN}x| = 0 = o^-_g(h).
\]
In particular, $I_g(h) = o^+_g(h) - o^-_g(h)$, and we are done.
\end{proof}

Theorem~\ref{thm:index} now follows immediately from Propositions~\ref{prop:orbit_number} and~\ref{prop:index_sum}.

\begin{rmk}
Although their difference is additive, individually the orbit numbers $o^+$ and $o^-$ do not behave so well under composition of elements of $\Fgp{g}$.  Take $(X,g)$ to be the binary odometer.  For $i \in \{0,1\}$ let $h_i$ act as $g^3$ on strings beginning with $i$ and $g\inv$ on strings not beginning with $i$.  Then it is easy to see that $o^+(h_0) = o^+(h_1)=1$ and $o^-(h_0) = o^-(h_1)= 0$, but $o^+(h_0h_1) = 3$ and $o^-(h_0h_1) = 1$.
\end{rmk}

\section{A normal form in terms of induced transformations}\label{sec:normal_form:induced}

We now move onto the proof of Theorem~\ref{thm:normal_form}.

\begin{proof}[Proof of Theorem~\ref{thm:normal_form}]
Let $(X,g)$ be a compact minimal system and let $I = I_g$.  We will use the following properties of $I$, which are clear from Lemma~\ref{lem:Cor_index} and Proposition~\ref{prop:index_sum}: It is a homomorphism from $\Fgp{g}$ to $\bZ$; $I(g_A)=1$ for all $A \in \mc{CO}^*(X)$; and $I(h) \ge 0$ for all $h \in \Fgpp{g}$ with equality if and only if $h = 1$.

Let $h \in \Fgp{g}$.  Set $r = \min_{x \in X}c_{g,h}(x)$.  We see that $\min_{x \in X}c_{g,hg^{-r}}(x) = 0$; thus $hg^{-r} \in \Fgpp{g}$ and the support $A_1$ of $hg^{-r}$ is a proper clopen subset of $X$.  Set 
\[
A_{i+1} = \supp(h_i) \text{ where } h_i = hg^{-r}g\inv_{A_1} g\inv_{A_2} \dots g\inv_{A_i}.
\]
Set $n = I(hg^{-r})$; since $hg^{-r} \in \Fgpp{g}$ we have $n \ge 0$.  For $i \le n$, by repeated application of Proposition~\ref{prop:po:cocycle}(iv) we see that $h_i \in \Fgpp{g}$; we also see that $I(h_i) = n-i$, so $h_i \neq 1$ for $i < n$ but $h_n = 1$.  Thus
\[
h = g_{A_n} \dots g_{A_2} g_{A_1}g^r.
\]
The construction also ensures that for all $1 \le i \le n$, the set $A_i$ is a nonempty clopen subset of $X$ such that $A_{i+1} \subseteq A_i$.

We must now show the uniqueness of the given expression for $h$.  Suppose that $h = g_{B_m} \dots g_{B_2} g_{B_1}g^s$, where the $B_i$ are proper nonempty clopen subsets of $X$ such that $B_{i+1} \subseteq B_i$ for all $i$.  Then $hg^{-s} \in \Fgpp{g}$ and is supported on $B_1 \neq X$, so $\min_{x \in X}c_{g,hg^{-s}} = 0$; consequently $\min_{x \in X}c_{g,h}(x) = s$, in other words $s = r$.  From now on we can consider expressions for $hg^{-r}$ rather than $h$, and so we may assume that $r = s = 0$.  Since $I$ is a homomorphism and every induced transformation of $g$ belongs to $(I)\inv(1)$, we see that $m= I(h) = n$.
 
Now proceed by induction on $|I(h)|$.  From the two expressions we have for $h$, easy calculations show that
\[
x \in A_1 \Leftrightarrow c_{g,h}(x) > 0 \Leftrightarrow x \in B_1,
\]
so $A_1 = B_1$.  We deduce that
\[
g_{A_n} \dots g_{A_2} = g_{B_n} \dots g_{B_2},
\]
and by the inductive hypothesis, we have $A_i = B_i$ for $2 \le i \le m$.

Lastly, note that $h \in \Fgpp{g}$ if and only if $\min_{x \in X}c_{g,h}(x) \ge 0$; as we have seen in the construction of the normal form, in fact $\min_{x \in X}c_{g,h}(x) = r$.
\end{proof}

If $X$ is zero-dimensional, we have another description of the topological conjugacy class of the compact minimal system $(X,g)$ directly in terms of the monoid structure of $\Fgpp{g}$.  Since the monoid $\Fgpp{g}$ has a more rigid structure than the group $\Fgp{g}$, this is more straightforward than the recovery of the flip-conjugacy class of $(X,g)$ from $\Fgp{g}$ as given by \cite{Rubin} or \cite{GPS2}.

\begin{prop}\label{prop:positive:reconstruction}
Let $(X,g)$ be a compact zero-dimensional minimal system and consider $\Fgpp{g}$ as a monoid.  Let $\mc{A}$ be the set of elements $h$ of $\Fgpp{g}$ that cannot be decomposed as $h=ab$ for $a,b \in \Fgpp{g} \smallsetminus \{1\}$.  Define the \defbold{support order} on $\mc{A}$ by setting $h_1 \le_{\supp} h_2$ if 
\[
\CC_{\mc{A} }(h_2) \subseteq \CC_{\mc{A} }(h_1) \cup \{h_2\}, \text{ where } \CC_{\mc{A} }(h) := \{a \in \mc{A}  \mid ah = ha \}.
\]
Then $(\mc{A} ,\le_{\supp})$ is a Boolean algebra with least element $\id_X$ and greatest element $g$, on which $\grp{g}$ acts by conjugation.  Moreover, $\mc{A}$ is $\grp{g}$-equivariantly isomorphic to the set $\mc{CO}(X)$ of compact open subspaces of $X$, ordered by inclusion.
\end{prop}

\begin{proof}
By Theorem~\ref{thm:normal_form} we see that $\mc{A} = \{g_A \mid A \in \mc{CO}(X)\}$.  By Lemma~\ref{lem:induced:conjugate}, given $A \in \mc{CO}(X)$ then $gg_Ag\inv = g_{gA}$, so $\grp{g}$ acts on $\mc{A}$ by conjugation.  It now suffices to show that $\le_{\supp}$ corresponds to the inclusion order on clopen sets in the obvious manner: that is, given $A,B \in \mc{CO}(X)$, that $g_A \le_{\supp} g_B$ if and only if $A \subseteq B$.

Suppose that $A \smallsetminus B$ is nonempty.  Then there is a proper nonempty clopen subset $A'$ of $A \smallsetminus B$.  Since $g_A$ acts minimally on $A$, it follows that $g_AA' \neq A'$; since $A' = \supp(g_{A'})$, this means that $g_A$ and $g_{A'}$ do not commute.  However, $g_{A'}$ does commute with $g_B$, since $A'$ and $B$ are disjoint.  Thus $g_A \not\le_{\supp} g_B$, proving that $g_A \le_{\supp} g_B \Rightarrow A \subseteq B$.

Now suppose instead that $A \subseteq B$ and suppose $B' \in \mc{CO}(X) \smallsetminus \{B\}$ is such that $g_B$ and $g_{B'}$ commute.  Then $g_B$ and $g_{B'}$ both preserve the set $B \cap B'$; since $B \cap B'$ cannot equal both $B$ and $B'$, by the minimality of $g_B$ and $g_{B'}$ we must have $B \cap B' = \emptyset$.  But then $A \cap B' = \emptyset$, so $g_A$ and $g_{B'}$ also commute.  Thus $g_A \le_{\supp} g_B$, proving that $A \subseteq B \Rightarrow g_A \le_{\supp} g_B$.
\end{proof}

By Stone duality, any compact zero-dimensional space can be recovered from its Boolean algebra of clopen subsets.  The following corollary is thus immediate.

\begin{cor}
Let $(X_1,g_1)$ and $(X_2,g_2)$ be compact zero-dimensional minimal systems, and suppose that $\theta: \Fgpp{g_1} \rightarrow \Fgpp{g_2}$ is an isomorphism of monoids.  Then $\theta(g_1) = g_2$ and there is a homeomorphism $\kappa: X_1 \rightarrow X_2$ such that $\kappa(hx) = \theta(h)(\kappa x)$ for all $x \in X_1$ and $h \in \Fgpp{g_1}$.
\end{cor}

Another consequence of our normal form for strongly positive elements is that it leads to a monoid presentation of $\Fgpp{g}$ in terms of induced transformations.

\begin{prop}\label{prop:monoid_presentation}
Let $(X,g)$ be a compact minimal system.  Then for all $A,B \in \mc{CO}^*(X)$, we have
\[
g_Ag_B = g_{A \ast B}g_{A \cup B},
\]
where $A \ast B$ is the following nonempty clopen subset of $A \cup B$:
\[
A \ast B := (A \cap g_{A \cup B}(B)) \cup ((B \smallsetminus A) \cap g_{A \cup B}(A \smallsetminus B)).
\]
With respect to the generating set $\{g_A \mid A \in \mc{CO}^*(X)\}$, $\Fgpp{g}$ has the monoid presentation
\[
\Fgpp{g} = \langle \{g_A \mid A \in \mc{CO}^*(X)\} \mid g_Ag_B = g_{A \ast B}g_{A \cup B} \rangle.
\]
\end{prop}

\begin{proof}
Note that $g_X = g$.  Given $A,B \in \mc{CO}^*(X)$, we see that $h = g_Ag_B$ is a strongly positive element such that $I(h)=2$ and $\supp(h) = A \cup B$.  The normal form of $h$ is therefore $h = g_{C}g_{A \cup B}$, for some nonempty clopen subset $C$ of $A \cup B$ depending on $A$ and $B$.

We can calculate $C$ as the support of $k = g_Ag_Bg\inv_{A \cup B}$, as follows.  Set
\[
D_1 = A \cap B \cap g_{A \cup B}(B); \; D_2 = (A \smallsetminus B) \cap g_{A \cup B}(B); \; D_3 = (B \smallsetminus A) \cap g_{A \cup B}(A \smallsetminus B);
\]
note that $D_1 \cup D_2 \cup D_3 = A \ast B$.
Take $x \in A \cup B$ and let $y = g\inv_{A \cup B}x$.  If $y \in A \smallsetminus B$ and $x \in A$, then $g_By = y = g\inv_Ax$, and hence $kx=x$; similarly $kx=x$ in the case that $y \in B$ and $x \in B \smallsetminus A$.  Thus $k$ fixes all points outside of $A \ast B$.  If $x \in D_1$, then $x,y \in B$, so $g_Bg\inv_{A \cup B}$ fixes $x$, and hence $kx = g_Ax \neq x$.  If $x \in D_2$ then $y \in B$ but $x \not\in B$, so $g_By >_g x$, and hence $kx >_g x$.  If $x \in D_3$, then $kx = g_Ay \in A$; since $x \not\in A$ we have $kx \neq x$.  Thus the support of $k$ is exactly $A \ast B$, showing that $C = A \ast B$ as claimed.

Let $\Gamma$ be the monoid with presentation
\[
\Gamma = \langle \{\gamma_A \mid A \in \mc{CO}^*(X)\} \mid \gamma_A\gamma_B = \gamma_{A \ast B}\gamma_{A \cup B} \rangle;
\]
since the corresponding relations are satisfied in $\Fgpp{g}$, we have a monoid homomorphism $\rho: \Gamma \rightarrow \Fgpp{g}$ such that $\rho(\gamma_A) = g_A$.  Since $\Fgpp{g}$ is generated by induced transformations, $\rho$ is surjective.  Now consider a word $w$ in the alphabet $\{\gamma_A \mid A \in \mc{CO}^*(X)\}$.  Since $A \ast B \subseteq A \cup B$ for all $A,B \in \mc{CO}^*(X)$, we can use the relations of $\Gamma$ to obtain a word $w'$ representing the same element of $\Gamma$ as $w$, of the form $w' = \gamma_{B_{n}} \dots \gamma_{B_2} \gamma_{B_1}$, where now $B_i \in \mc{CO}^*(X)$ such that $B_{i+1} \subseteq B_i$ for all $i$.  We declare the result of this process to be a reduced word.

By Theorem~\ref{thm:normal_form}(i), $\rho$ sends distinct reduced words to distinct elements of $\Fgpp{g}$, so $\rho$ is injective and hence an isomorphism.
\end{proof}

Given Proposition~\ref{prop:po:cocycle}(ii), we see that the monoid presentation of $\Fgpp{g}$ is also a group presentation for $\Fgp{g}$.

\section{Strongly \ppm homeomorphisms}\label{sec:ppm}
A zero-dimensional space $X$ is \defbold{h-homogeneous} if every nonempty clopen subset of $X$ is homeomorphic to $X$ itself; see \cite{SR}\footnote{The theme of \cite{SR} is homogeneous Boolean algebras; the Stone space of such an algebra is then an example of an h-homogeneous space.} and \cite{Terada} for some properties and examples.  A \defbold{minimal h-homogeneous Stone space} is an infinite topological space $X$ that is compact, Hausdorff and h-homogeneous, and admits a minimal homeomorphism; the principal example we are interested in is when $X$ is the Cantor space, but the arguments apply equally well to other examples.

Say that a homeomorphism $h$ of a compact Hausdorff space $X$ is \defbold{piecewise a power of a minimal homeomorphism} (\ppm) if $X = \emptyset$ or there exists $g \in \Homeo(X)$ such that $g$ is minimal on $X$ and $h \in \Fgp{g}$.  Fix such a homeomorphism $h \in \Homeo(X)$ for $X \neq \emptyset$ and consider the minimal homeomorphisms $g \in \Homeo(X)$ such that $h \in \Fgp{g}$.  We can take the orbit number $o_g(h)$ as a measure of the `efficiency' with which $g$ witnesses that $h$ is \ppm; define $o_{\min}(h)$ to be $0$ if $X$ is finite, and otherwise the smallest value of $o_g(h)$, as $g$ ranges over all minimal homeomorphisms of $X$ such that $h \in \Fgp{g}$.  Letting $m(h)$ be the number of distinct infinite minimal orbit closures $X_1,\dots,X_m$ of $h$ on $X$, we see that $o_g(h) = \sum^{m}_{i=1}o_g(h_i)$, where $h_i$ is the restriction of $h$ to $X_i$, and $o_g(h_i) \ge 1$ for each such restriction; hence $o_{\min}(h) \ge m(h)$.  We say that $h$ is \defbold{strongly \ppm} if $o_{\min}(h) = m(h)$.  In this section, we characterize the strongly \ppm homeomorphisms of minimal h-homogeneous Stone spaces.

Not all finite-order homeomorphisms $h$ of the Cantor space are \ppm, because the set of fixed points of $h$ can be any closed set, not necessarily an open set.  Given $X = \{0,1\}^*$ and a closed but not open subset $Y$ of $X$, there is an infinite sequence of finite binary strings $(w_i)_{i \ge 0}$ such that $X \smallsetminus Y$ is the disjoint union of the sets $w_iX$, and writing $\overline{e}:=1-e$ for $e \in \{0,1\}$, there is a homeomorphism $h$ of $X$ of order $2$ such that $h$ fixes $Y$ pointwise and $h(w_iev) = w_i\overline{e}v$ for all $i \ge 0$, $e \in \{0,1\}$ and $v \in X$.  Similar examples can be constructed for any minimal h-homogeneous Stone space.

The next proposition reduces the study of strongly \ppm homeomorphisms of minimal h-homogeneous Stone spaces to the aperiodic case: in particular, it implies that $h$ is strongly \ppm in the case that $h$ has a minimal-periodic partition and $m(h) \le 1$.

\begin{prop}\label{prop:remove_periodic}
Let $X$ be a minimal h-homogeneous Stone space and let $h \in \Homeo(X)$.  Suppose that $X_\per$ is clopen and $(X_\per,h)$ is strongly periodic, and that $(X_{\aper},h)$ is \ppm (allowing $X_\aper=\emptyset$).  Then $h$ is \ppm and $o_{\min}(h) = o_{\min}(h_\aper)$.
\end{prop}

\begin{proof}
As a warm-up for the general argument, let us first consider the case that every orbit of $h$ has the same finite size $n \ge 1$.  Then we can partition $X$ into clopen sets $X(i)$ where $0 \le i < n$, $hX(i) = X(i+1)$ and $hX(n-1) = X(0)$.  (An equivalent construction implicitly occurs in \cite{Krieger}; for more detail see for example \cite[Lemma~2.10]{GR}.)  The set $X(0)$ is homeomorphic to $X$ and hence there is a minimal $g' \in \Homeo(X(0))$.  We then define a homeomorphism $g$ on $X$, as follows:
\[
gx = 
\begin{cases}
h x &\mbox{if} \;  x \in X(i),\; 0 \le i < n-1\\
g' h^{1-n}x &\mbox{if} \;  x \in X(n-1).
\end{cases}.
\]
We then see that $g^n$ acts as the minimal homeomorphism $g'$ on $X(0)$, so each forward orbit of $g$ has dense intersection with $X(0)$.  Using the other powers of $g$ it follows that the orbits of $g$ on $X$ are dense, that is, $(X,g)$ is a compact minimal system.  From the construction it is clear that $h \in \Fgp{g}$; since $h$ has finite order, $o_g(h)=0$.

We now turn to the general case of the proposition.  We know that either $X_\aper$ is empty or that $h_{X_\aper} \in \Fgp{g'}$, where $g'$ is a minimal homeomorphism of $X_\aper$.  Partition $X_\per$ into clopen sets $X_\per(n,i)$, where $0 \le i < n$, $hX_\per(n,i) = X_\per(n,i+1)$ and $hX_\per(n,n-1) = X_\per(n,0)$.  These sets are nonempty for only finitely many $n \in \bN$; let us say that
\[
\{ X_\per(n,i) \mid n \in \bN, 0 \le i < n, X_\per(n,i) \neq \emptyset\} = \{Y_1,\dots,Y_l\}.
\]
We arrange the sets $Y_j$ so that if $Y_j = X_\per(n,i)$ and $Y_{j'} = X_\per(n',i')$, then $j < j'$ if and only if either $n < n'$ or $n=n'$ and $i < i'$.

Since $X$ is h-homogeneous, the sets $Y_1,\dots,Y_l$ are all homeomorphic to $X$; if $X_\aper$ is nonempty then $X_\aper$ is homeomorphic to $Y_1$.  If $X_\aper$ is nonempty, write $X_\aper = Y_0$ and set $\epsilon=0$; if $X_\aper$ is empty, set $\epsilon=1$ and choose a minimal $g' \in \Homeo(Y_1)$.  Choose homeomorphisms $t_j: Y_j \rightarrow Y_{j+1}$ for $\epsilon \le j < l$.  In the case that $Y_j = X_\per(n,i)$ and $Y_{j+1} = X_\per(n,i+1)$ for some $n$ and $i$, then we define $t_j$ by setting $t_jx = hx$ for all $x \in Y_j$; otherwise choose an arbitrary homeomorphism.

Now define a homeomorphism $g$ on $X$, as follows:
\[
gx = 
\begin{cases}
t_j x &\mbox{if} \;  x \in Y_j, \; \epsilon \le j < l\\
g' t\inv_\epsilon t\inv_{\epsilon+1} \dots t\inv_{l-1}x &\mbox{if} \;  x \in Y_l
\end{cases}.
\]

Observe that for all $x \in Y_{\epsilon}$, we have $g^{l+1-\epsilon}x = g'x$, so $g'_{Y_{\epsilon}} \in \Fgp{g}$ and in particular $h_{\aper} \in \Fgp{g}$.  Since $g'$ is a minimal homeomorphism of $Y_{\epsilon}$, it follows that each forward $g$-orbit has dense intersection with $Y_{\epsilon}$, and then the other powers of $g$ ensure that $g$ acts minimally on the whole of $X$.  When $x \in X_\per(n,i)$ for $0 \le i < n-1$, then $hx = gx$; when $x \in X_\per(n,n-1)$, then $hx = g^{1-n}x$.  Thus $h \in \Fgp{g}$, showing that $h$ is p.p.m.  The construction ensures that $o_g(h) = o_{g'}(h_\aper)$, and hence $o_{\min}(h) \le o_{\min}(h_\aper)$.  On the other hand if $k \in \Homeo(X)$ is any minimal homeomorphism such that $h \in \Fgp{k}$, then $h_{\aper} \in \Fgp{k_A}$ where $A = X^{\grp{h}}_{\aper}$ and $o_{k}(h) = o_{k_A}(h_{\aper}$, so $o_{\min}(h) \ge o_{\min}(h_\aper)$.
\end{proof}

A minimal homeomorphism is clearly Kakutani equivalent to its induced transformations on nonempty clopen sets.  In particular, if $(X,g)$ is a compact minimal system and $\{X_1,\dots,X_n\}$ is a partition of $X$ into nonempty clopen sets, then the compact minimal systems $(X_i,g_{X_i})$ lie in a single Kakutani equivalence class.  In fact, all finite tuples of Kakutani equivalent systems arise in this way.

\begin{prop}\label{prop:Kakutani_weld}
Let $n$ be a natural number and let $(X_i,g_i)_{1 \le i \le n}$ be an $n$-tuple of Kakutani equivalent compact minimal systems.  Then there is a minimal homeomorphism $g$ of the disjoint union $X = \bigsqcup^n_{i=1}X_i$, such that $g_{X_i} = g_i$ for $1 \le i \le n$.
\end{prop}

\begin{proof}
Suppose $n \ge 3$ and that the proposition is true for all smaller choices of $n$.  Then there is a minimal homeomorphism $g'$ of $X' = \bigsqcup^{n-1}_{i=1}X_i$ such that $g'_{X_i} = g_i$ for $1 \le i \le n-1$.  Now $g'$ is Kakutani equivalent to $g_n$, so there is a minimal homeomorphism $g$ of $X = X' \sqcup X_n$ such that $g_{X'} = g'$ and $g_{X_n} = g_n$.  We then see that $g_{X_i} = g_i$ for $1 \le i \le n$.  Thus it suffices to prove the result for $n \le 2$.  Since the case $n=1$ is trivial, we assume $n=2$.

Choose a Kakutani equivalence $\kappa: Y_1 \rightarrow Y_2$ of $(g_1,g_2)$, where $Y_i$ is a nonempty clopen subset of $X_i$.  By restricting $\kappa$, we may ensure that $g_iY_i$ is disjoint from $Y_i$.  Now let $X = X_1 \sqcup X_2$ and define $g: X \rightarrow X$ by setting
\[
gx = 
\begin{cases}
g_2\kappa x &\mbox{if} \;  x \in Y_1\\
g_1(g_1)\inv_{Y_1} \kappa\inv x &\mbox{if} \;  x \in Y_2 \\
g_1x &\mbox{if} \; x \in X_1 \smallsetminus Y_1 \\
g_2x &\mbox{if} \; x \in X_2 \smallsetminus Y_2
\end{cases}.
\]
We now compare $g_{X_i}$ with $g_i$.  If $x \in X_i \smallsetminus Y_i$, then certainly $g_{X_i}x = g_ix$.  If $x \in Y_{3-i}$, then the sequence $gx,g^2x,\dots$ first passes through $g_{i}Y_{i}$, then follows the forward $g_{i}$-orbit until it reaches $Y_{i}$, then at the next step moves to $g_{3-i}Y_{3-i}$.  The first point on the forward $g_{i}$-orbit in $Y_i$ after visiting $g_{i}Y_{i}$ is given by applying $(g_{i})_{Y_{i}}g\inv_{i}$.

The result is as follows: if $x \in Y_1$, then
\begin{align*}
g_{X_1}x &= (g_1 (g_1)\inv_{Y_1} \kappa\inv)((g_2)_{Y_2}g\inv_2)g_2\kappa x = g_1  (g_1)\inv_{Y_1} \kappa\inv (g_2)_{Y_2} \kappa x = \\
 &=  g_1(g_1)\inv_{Y_1} \kappa\inv  \kappa (g_1)_{Y_1}  x  = g_1x.
\end{align*}
If $x \in Y_2$, then
\begin{align*}
g_{X_2}x &= (g_2 \kappa) ((g_1)_{Y_1}g\inv_1) (g_1 (g_1)\inv_{Y_1} \kappa\inv) x = g_2x.
\end{align*}
Thus for all $1 \le i \le 2$ and all $x \in X_i$, we have $g_{X_i}x = g_ix$, proving that $g_{X_i} = g_i$.  We see from the construction that $g$ is bijective and is a local homeomorphism, so $g$ is a homeomorphism.  By the minimality of $g_{X_i}$ on $X_i$, any nonempty closed $g$-invariant set contains $X_1$ or $X_2$; since $gY_1 = Y_2$, it follows that $g$ is in fact minimal.
\end{proof}

We can now characterize the strongly \ppm property in terms of flip Kakutani equivalence.

\begin{prop}\label{prop:Kakutani_weld:bis}
Let $X$ be a compact Hausdorff space and let $h \in \Homeo(X)$.  Suppose that $h$ admits a minimal-periodic partition, and that either $h$ is aperiodic or $X$ is a minimal h-homogeneous Stone space.  Then $h$ is strongly \ppm if and only if either $m(h)=0$, or $m(h) \ge 1$ and the spaces $(Y,h)$ for $Y$ an infinite orbit closure of $h$ all lie in a single flip Kakutani equivalence class.
\end{prop}

\begin{proof}
By Proposition~\ref{prop:remove_periodic}, we may suppose that $h$ is aperiodic.  Write $X_1,\dots,X_m$ for the distinct infinite orbit closures of $h$.

If the systems $(X_i,h)$ are all flip Kakutani equivalent, choose a homeomorphism $h_i$ on $X_i$, so that the systems $(X_i,h_i)$ are all Kakutani equivalent and $h$ acts as either $h_i$ or $h\inv_i$ on $X_i$.  By Proposition~\ref{prop:Kakutani_weld} there is a minimal homeomorphism $g$ of $X$ such that $h_i = g_{X_i}$; in particular, we can regard $h_i$ as an element of $\Fgp{g}$ with $o_g(h_i) = 1$.  It is then clear that $h \in \Fgp{g}$, with $o_g(h) = \sum^m_{i=1}o_g(h_i) = m$.

Conversely, suppose that $h$ is strongly \ppm, that is, there is a minimal homeomorphism $g$ such that $h \in \Fgp{g}$ with $o_g(h) = m$.  Let $h_i$ be the restriction of $h$ to $X_i$.  Then $o_g(h_i) \ge 1$ for each $i$, since there exist infinite orbits of $h_i$, and 
\[
m = o_g(h) = \sum^m_{i=1}o_g(h_i).
\]
We therefore have $o_g(h_i)=1$ for all $i$.  By Theorem~\ref{thm:sign} we see that $h^{\epsilon_i}_i \in \Fgpnn{g}$ for $\epsilon_i \in \{-1,1\}$, and by Theorem~\ref{thm:positive_conjugation}, $h^{\epsilon_i}_i$ is conjugate in $\Fgp{g}$ to some $h'_i \in \Fgpp{g}$.  Thus $h'_i$ is an aperiodic homeomorphism of $X_i$ that is strongly positive, with $o_g(h'_i)=1$; the only possibility is that $h'_i = g_{X_i}$.  Thus $(X_i,h_i)$ is flip Kakutani equivalent to $(X,g)$.
\end{proof}

\end{document}